\documentclass[12pt]{amsart}

\usepackage{biblatex}
\usepackage{pgf,tikz,pgfplots,enumerate}
\pgfplotsset{compat=1.17}
\usepackage{mathrsfs}
\usetikzlibrary{arrows}

\usepackage{amssymb}
\usepackage{kbordermatrix} 
\renewcommand{\kbldelim}{[} 
\renewcommand{\kbrdelim}{]}

\newcommand{\sgn}{\operatorname{sgn}}
\newcommand{\rank}{\operatorname{rank}}
\newcommand{\Isom}{\operatorname{Isom}}
\newcommand{\Full}{\operatorname{Full}}

\newcommand{\bR}{{\mathbb{R}}}

\title{Braced triangulations and rigidity}

\author{James Cruickshank}
\email{james.cruickshank@nuigalway.ie}
\address{Sch. of Math., Stats. \& Appl. Math.,  
NUI Galway, Ireland.}

\author[E. Kastis]{Eleftherios Kastis}
\email{l.kastis@lancaster.ac.uk}
\address{Dept.\ Math.\ Stats.\\ Lancaster University\\
Lancaster, LA1 4YF \\U.K. }

\author[D. Kitson]{Derek Kitson}
\email{derek.kitson@mic.ul.ie}
\address{Dept.\ Math.\ Comp. St.\\Mary Immaculate College, Thurles, Co.~Tipperary, Ireland.}

\author{Bernd Schulze}
\email{b.schulze@lancaster.ac.uk}
\address{Dept.\ Math.\ Stats.\\ Lancaster University\\
Lancaster, LA1 4YF \\U.K. }

\thanks{E.K. and D.K. supported by the Engineering and Physical Sciences Research Council [grant number EP/S00940X/1].}

\usepackage{a4wide}

\newtheorem{thm}{Theorem}
\newtheorem{lem}[thm]{Lemma}
\newtheorem{cor}[thm]{Corollary}

\newtheorem{ex}[thm]{Example}
\newtheorem{prop}[thm]{Proposition}
\newtheorem{rem}[thm]{Remark
}

\newtheorem*{lem*}{Lemma}


\begin{document}

\maketitle

\begin{abstract}
    We consider the problem of finding an inductive construction,
    based on vertex splitting,
    of triangulated 
    spheres with a fixed number  of additional edges (braces). 
    We show that for any positive integer \( b \) there is such 
    an inductive construction of triangulations with \( b \) braces, 
    having finitely many base graphs.
    In particular we establish a bound for the maximum size of a base graph with $b$ braces that is linear in $b$.
    In the case that $b=1$ or $2$ we determine the list of base graphs explicitly. Using these results we show that doubly braced triangulations are (generically) minimally rigid in two distinct geometric contexts arising from a hypercylinder in $\mathbb{R}^4$ and a class of mixed norms on $\mathbb{R}^3$.
\end{abstract}

\section{Introduction}

A $d$-dimensional \emph{bar-joint framework} is a pair
$(G,q)$, where $G=(V,E)$ is a simple graph and $q\in (\mathbb{R}^d)^V$. 
We think of a framework as a collection of (fixed length) bars that are connected at their
ends by (universal) joints. Loosely speaking, such a  framework  is called \emph{rigid} if it cannot be deformed continuously into another non-congruent
framework while preserving the lengths of all bars. Otherwise, the framework is called \emph{flexible}. 

The rigidity and flexibility analysis of bar-joint frameworks and related constraint systems has a rich history which dates back to the work of Euler and Cauchy on the rigidity of polyhedra and to Maxwell's studies of mechanical linkages and trusses in the 19th century. Over the last few decades, the field of geometric rigidity theory has seen significant developments due to a plethora of new applications in pure mathematics and diverse areas of science, engineering and design. 
We refer the reader to \cite{WMatroid,SWHandbook}, for example, for summaries of key definitions and results.

Triangulations of the 2-sphere play an important role in the rigidity theory of bar-joint frameworks. Gluck has shown that generic realisations of these graphs as bar-joint frameworks in 3-dimensional Euclidean space are minimally rigid (\cite{Gluck}). Whiteley gave an independent proof of
Gluck's result by observing that certain vertex splitting moves preserve generic rigidity and are sufficient to construct all sphere triangulations from an easily understood base graph (\cite{WWsplitting}). 

In this paper we consider inductive constructions, based on vertex splitting, for triangulated 
    spheres with a fixed number  of additional edges (braces). 
Our first main result 
establishes a linear bound for the size of an irreducible (defined below) braced triangulation with $b$ braces (Theorem \ref{thm_main_bound}). An easy consequence is that for fixed $b$ there are only finitely many irreducible braced triangulations.
This is analogous to well known results on irreducible triangulations of surfaces by Barnette, Edelson, Boulch, Nakamoto, Colin de Verdi\'ere and others (\cites{BarnetteEdelson,boulch}). 

The case $b=1$ is quickly dealt with in Section \ref{s:unibrace} where we show that a triangular bipyramid with a brace connecting the two poles is the unique irreducible. 
In other words, every unibraced triangulation can be constructed from this single irreducible by a sequence of vertex splitting moves of a specific kind. 
Triangulated spheres with a single brace have previously been studied in \cite{Whi} in relation to redundant rigidity and more recently in \cite{CJ,JordanTanigawa} in relation to global rigidity. 
The case $b=2$ is more involved and in Section \ref{sec_b2}
we show that there are exactly five distinct irreducibles (see Figures \ref{fig_doubly_braced_oct}, \ref{fig_hex_disjoint}, \ref{fig_hex_adjacent}, \ref{fig_ireed_seven}, \ref{fig_ireed_seven_non}).

Two major new research strands in geometric rigidity are the rigidity analyses of bar-joint frameworks in Euclidean $3$-space whose joints are constrained to move on a surface (such as a cylinder or surface of revolution) \cite{jkn,jn,nop,nop1} and of bar-joint frameworks in non-Euclidean normed spaces \cite{dkn,kitson,kitlev,kit-pow,kit-sch}.

In Section \ref{sec:hypercylinder} 
we 
prove an analogue of Gluck's Theorem for bar-joint frameworks which are constrained to a hypercylinder in $\mathbb{R}^4$ (Theorem \ref{t:hypercylinder}). In this setting it is clear that
doubly braced triangulations have exactly the right number of edges to be minimally rigid and so our inductive construction from Section \ref{sec_b2} is a key ingredient in the proof. We introduce the appropriate rigidity matrix for frameworks on the hypercylinder, construct rigid placements for the base graphs and show that vertex splitting preserves rigidity on the hypercylinder.

In Section \ref{sec:mixednorms} we prove another analogue of Gluck's Theorem, this time for a class of {\em mixed norms} on $\mathbb{R}^3$ (Theorem \ref{t:norms}). In this setting we first need to establish some key geometric properties of the underlying normed spaces, in particular we characterise the isometries of the spaces. Our inductive construction for doubly braced triangulations is again key to the proof.

\section{Braced triangulations}

A sphere graph is a simple  graph with a fixed embedding in the $2$-sphere without edge crossings.
A {\em face} of a sphere graph is the topological closure of a connected component of the complement of the graph in the sphere. In particular a face contains its boundary.

A (sphere) triangulation, $P$, is a maximal sphere graph with at least 3 vertices.
We say that an edge $e \in E(P)$ is contractible in 
$P$ if it belongs to precisely two 3-cycles. In other words it does not belong to
any non-facial 3-cycle of $P$. In that case it follows that the simple graph $P/e$ obtained by contracting the edge $e$ is also
a triangulation (with the obvious embedding).

The following two lemmas are well known and will be useful for us. See, for example  \cite{DiwanKurhekar}.

\begin{lem}
    \label{lem_vertex_incident_good_edge}
    Suppose that \( P \) is a triangulation with at least \( 4 \) vertices
    and that \( F \) is a face of \( P \). 
    Each vertex of \( F \) is incident to a contractible edge of \( P \) that is not in \( F \). \qed
\end{lem}

\begin{lem} 
    \label{lem_good_edge_not_in_face} 
    Suppose that \( P \) is a triangulation with at least \( 4 \) vertices.
    Each vertex of \( P \) is incident to at least two contractible edges
    of \( P \). \qed
\end{lem}

A {\bf braced triangulation} is a simple graph $G$ obtained from a triangulation $P$ by adjoining additional edges to $P$. We refer to these additional edges as {\em braces}. 
We write \(G= (P,B) \) to denote the braced triangulation obtained from a triangulation \( P \) and by adjoining the set $B$ of braces. 
We denote by $V(B)$ the set of vertices in $G$ which are incident with a brace.
An edge \( e \) of \( P \) is said to be {\em contractible} in \( G \) if 
\( e \) is contractible in \( P \) and
\( e \) does not belong to any $3$-cycle that contains
a brace. So in that case \( G/e = (P/e,B) \) is also a 
braced triangulation.

A braced triangulation \( G = (P,B) \)
is said to be {\em irreducible} if there is no edge of \( P \) that is contractible
in \( G \). This is analogous to the notion of irreducible triangulation that is well studied in the literature on 
triangulations of surfaces. In that context, it is known that for a given surface 
there are finitely many isomorphism classes of irreducible triangulations - 
see \cite{BarnetteEdelson} and \cite{boulch}.
In this section, we will show that 
for a given number of braces there are  finitely many isomorphism classes of
irreducible braced triangulations.

For a vertex \( v \) of \( P \) we write \( N_P(v) \) for the 
set of neighbours of \( v \) in \( P \). That is to say 
\( N_P(v) = \{ u \in V(P): uv \in E(P)\} \).
For vertices \( u,v \) let \( X_{uv} = N_P(u)\cap N_P(v)  \)
and define $r_{uv} = |X_{uv}|$.

\begin{lem}
    \label{lem_vertices}
    Suppose that \( G=(P,B) \) is an irreducible braced triangulation. Then 
    \begin{equation} \label{eq_vertices} 
    V(P) = V(B) \cup \bigcup_{uv \in B} X_{uv}  \end{equation}
\end{lem}
\begin{proof}
    Suppose that \( w \in V(P)\backslash V(B) \). Then by Lemma \ref{lem_vertex_incident_good_edge} 
    there is some edge \( xw \) in \( P \) such that \( xw \) is contractible in 
    \( P \). Since \( G \) is irreducible, it follows that there is some
    brace \( xy \in B \) such that \( yw \in E(P) \). Clearly 
    \( w \in X_{xy} \) as required.
\end{proof}

Now suppose that \( uv \in B \).
Let \( Q_{uv} \) be the sphere subgraph of \( P \) that is formed by the 
complete bipartite graph \( K(\{u,v\},X_{uv}) \). 
 We will use this graph frequently
in the sequel, so we label its various elements as follows (see Figure \ref{fig_of_Q}).
Suppose $|X_{uv}|\geq 2$. Let \( R_1,\ldots,
R_{r_{uv}}\) be the faces of \( Q_{uv} \) with the labels chosen so that \( R_{i} \) is 
adjacent to \( R_{i+1} \) for \( i = 1,\ldots,r_{uv} \). Here we adopt the convention that 
\( R_{r_{uv}+1} = R_1 \). 
We suppose that the boundary vertices of \( R_i \) are \( y_{i},u,y_{i+1},v \) for \( i = 1,\ldots,r_{uv} \). 
So \( X_{uv} = \{y_1,\ldots,y_{r_{uv}}\} \) and by convention \( y_{r_{uv}+1} = y_1 \). 
Note that if \( y  \) is a point in the sphere 
and \( y \neq u,v \) then \( y \) belongs to at most two of \( R_1,\ldots,R_{r_{uv}} \). 
Furthermore if \( y \in V(P) - \{u,v\} \) then \( y \) belongs to exactly two of 
\( R_1,\ldots,R_{r_{uv}} \) if and only if \( y \in X_{uv} \).

    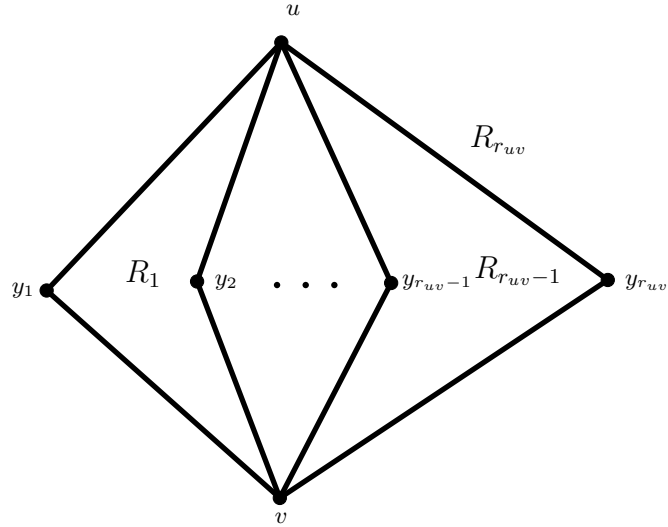
\begin{figure}[ht]
        \definecolor{ududff}{rgb}{0,0,0}
\begin{tikzpicture}[line cap=round,line join=round,>=triangle 45,x=1cm,y=1cm]
\draw [line width=2pt] (-7.87,6.64)-- (-10.99,3.34);
\draw [line width=2pt] (-7.87,6.64)-- (-8.99,3.46);
\draw [line width=2pt] (-7.87,6.64)-- (-6.41,3.44);
\draw [line width=2pt] (-7.87,6.64)-- (-3.53,3.48);
\draw [line width=2pt] (-7.89,0.58)-- (-3.53,3.48);
\draw [line width=2pt] (-7.89,0.58)-- (-6.41,3.44);
\draw [line width=2pt] (-7.89,0.58)-- (-8.99,3.46);
\draw [line width=2pt] (-7.89,0.58)-- (-10.99,3.34);
\draw (-8.19,3.6) node[anchor=north west] {\huge $\mathbf \dots$};
\draw (-10.09,3.88) node[anchor=north west] {$R_1$};
\draw (-5.45,3.94) node[anchor=north west] {$R_{r_{uv}-1}$};
\draw (-5.51,5.68) node[anchor=north west] {$R_{r_{uv}}$};
\begin{scriptsize}
\draw [fill=ududff] (-7.87,6.64) circle (2.5pt);
\draw[color=ududff] (-7.71,7.07) node {$u$};
\draw [fill=ududff] (-10.99,3.34) circle (2.5pt);
\draw[color=ududff] (-11.3,3.34) node {$y_1$};
\draw [fill=ududff] (-8.99,3.46) circle (2.5pt);
\draw[color=ududff] (-8.6,3.45) node {$y_2$};
\draw [fill=ududff] (-6.41,3.44) circle (2.5pt);
\draw[color=ududff] (-5.8,3.45) node {$y_{r_{uv}-1}$};
\draw [fill=ududff] (-3.53,3.48) circle (2.5pt);
\draw[color=ududff] (-3,3.45) node {$y_{r_{uv}}$};
\draw [fill=ududff] (-7.89,0.58) circle (2.5pt);
\draw[color=ududff] (-7.87,0.31) node {$v$};
\end{scriptsize}
\end{tikzpicture}

        \centering
        \caption{The faces of \( Q_{uv} \). Each \( R_i \) is a closed quadrilateral region in the sphere.}
        \label{fig_of_Q}
    \end{figure}

\begin{lem}
    \label{lem_face_with_no_braced_vertices}
    Suppose that \( G=(P,B) \) is irreducible, \( uv \in B \), 
    and that some face of \( Q_{uv} \) contains no 
    vertices in \( V(B) \) other than \( u,v \). Then 
    \( r_{uv} \leq 3 \).
\end{lem}

\begin{proof}
    Suppose $r_{uv}\geq 4$ and suppose that \( R \) is a face of \( Q_{uv} \) satisfying the hypothesis. 
    Let \( a,u,b,v \) be the boundary vertices 
    of \( R \).
    First we claim that there are no vertices of \( P \)
    in the interior of \( R \). If \( w \)  was such a vertex then
    by Lemma \ref{lem_vertex_incident_good_edge} it is incident with an edge which is 		contractible in $P$. Since $w\notin V(B)$, it follows there is some
    brace \( xy \in B \) such that \( wx  \) and \( wy \)
    are edges in \( P \). Now since, 
    by assumption, \( w \not \in X_{uv} \), at least one of 
    \( x,y \), wlog say \( x \), is not in \( \{u,v\} \). 
    Then \( x \in R \) and \( x \in V(B) -\{u,v\} \) 
    contradicting our hypothesis.

    So, in view of this claim, and since \( P \) is 
    a triangulation that does not contain the edge \( uv \), 
    it follows that \( ab \) is an edge of \( P \) that is 
    contained within \( R \). Now since neither \( a \) 
    nor \( b \) 
    is in \( V(B) \), it follows that \( ab \) is in some 
    non-facial 3-cycle of \( P \). Let 
    \( c \) be the third vertex of this 3-cycle. Clearly 
    \( c \not\in\{u,v\} \).
    Now let \( S,T \) be the distinct faces of \( Q_{uv} \) that 
    are adjacent to \( R \). Clearly \( c \in S \cap T \). 
    On the other hand, since \( r_{uv} \geq 4 \) it follows that \( S \cap T =
    \{u,v\}\), a contradiction.
\end{proof}

\begin{thm}
\label{thm_quad_bound}
Suppose that $G= (P,B)$ is an irreducible braced triangulation and that $b = |B| \geq 2$. Then $|V(P)| \leq 4b^2- 2b$.
\end{thm}

\begin{proof}
By Lemma \ref{lem_vertices} we know that $|V(P)| \leq 2b + \sum_{uv \in B} r_{uv}$. 
Now if $r_{uv} \geq 4$ then it follows from Lemma \ref{lem_face_with_no_braced_vertices} that every face of $Q_{uv}$ contains an element of $V(B) - \{u,v\}$. Since any such element belongs to at most two faces of $Q_{uv}$ we easily conclude that $r_{uv} \leq 4b-4$. Thus $\sum_{uv \in B} r_{uv} \leq b(4b-4)$ and the result follows.  
\end{proof}

We have the following immediate corollary of Theorem
\ref{thm_quad_bound}.
\begin{cor}
\label{cor:finite}
    For any positive integer \( b \),
    there are finitely many irreducible braced triangulations with \( b \)
    braces.
\end{cor}

It is natural to wonder if the bound in Theorem 
\ref{thm_quad_bound} can be sharpened. Indeed 
in the context of triangulations of surfaces of positive genus Boulch et al. in \cite{boulch} have established that 
if $f(g,c)$ is the maximum size of an irreducible triangulation of a surface with genus $g$ and $c$ boundary components, then $f(g,c)$ is
$\mathcal O(g+c)$. Motivated by this we devote the remainder of this section to establishing the following linear bound for the number of vertices of an irreducible braced sphere triangulation in terms of the number of braces.

\begin{thm}
    \label{thm_main_bound}
    Suppose that \( G = (P,B) \) is an irreducible braced triangulation.
    Then \( |V(P)| \leq 11b -4\).
\end{thm}

Before giving the proof of Theorem \ref{thm_main_bound} we need some lemmas.

\begin{lem}
    \label{lem_ruv_5}
    Suppose that $G=(P,B)$ is a braced triangulation
    such that $uv \neq xw$ and  $r_{uv},r_{xw}\geq 4$. Then either 
    \begin{enumerate}[(i)]
        \item  there exists a face $R$ of $Q_{uv}$ that contains $Q_{xw}$, or,
        \item $r_{uv} = r_{xw} = 4$ and $Q_{uv}\cup Q_{xw}$ is an octahedral graph.
    \end{enumerate}
\end{lem}

\begin{proof}
Suppose that one of $x,w$, say $x$, is contained in the interior of some face $R$ of $Q_{uv}$.
Then $N_P(x) \subset R$ and since at least one of $u,v$ is not in $N_P(x)$ and $r_{xw} \geq 4$ it follows that $w \in R$ also. So $Q_{xw} \subset R$ in this case. 

So, using the fact that $xw \not\in E(P)$ we can assume that $\{x,w\} \subset X_{uv}$.
If $\{x,w\}$ is contained in a face $R$ of $Q_{uv}$ then since $r_{uv} \geq 4$ it follows that $Q_{xw} \subset R$ and we are done. On the other hand suppose that there is no face of $Q_{uv}$ containing both of $x,w$. Then it follows that $X_{xw} \subset V(Q_{uv})$ and using the assumption that $r_{uv},r_{xw} \geq 4 $ we see that the only possibility is that $Q_{uv} \cup Q_{xw}$ is an octahedral graph.
\end{proof}

Now suppose that $G = (P,B)$ is a braced sphere triangulation. For the remainder of this section it will be convenient to work in the context of plane graphs instead of sphere graphs. So we fix some point in the sphere that is not in $P$ and by removing that point we consider $P$ as a plane graph. In particular any subgraph of $P$ has a unique unbounded face. 

We will need the following elementary observations about certain collections of plane graphs.

Suppose that $C$ is a finite set of plane graphs such that 
\begin{equation}
    \label{eqn_planegraphs}
\text{for all $H,K \in C$ with $H \neq K$, there is some face of $H$ that contains $K$.}
\end{equation}
Observe that $C$ has a partial order defined by 
$H \preceq K$ if either $H=K$ or $H$ is contained in a bounded face of $K$. 
This partial order will be key in the remainder of this section.

\begin{lem}
\label{lem_ordering2}
Suppose that $H,K$ are distinct graphs in $C$ and that there is some $z$ in the plane that does not lie in the unbounded face of $H$ and does not lie in the unbounded face of $K$. Then either $H \prec K$ or $K \prec H$.
\end{lem}

\begin{proof}
Suppose that $H,K$ are incomparable. Then, by (\ref{eqn_planegraphs}), $H$ is contained in the unbounded face of $K$ and vice versa. Since $z$ lies in a bounded face of $H$, it must lie in the unbounded face of $K$ which contradicts our assumption.
\end{proof}

We have the following immediate consequence of Lemma \ref{lem_ordering2}.

\begin{cor}
\label{cor_totalorder}
For any point $z $ in the plane let $$C^z = \{H \in C: \text{$z$ is not in the unbounded face of $H$}\}.$$ Then, with respect to $\preceq$, $C^z$ is a totally ordered subset of $C$. \qed
\end{cor}

From now on, let $C_G = \{Q_{uv}: uv \in B, r_{uv} \geq 6\}$ and suppose $|C_G| = c$. 
For convenience we will write $C_G= \{Q_1,\dots,Q_c\}$ where $Q_i = Q_{u_iv_i}$ and $r_i = r_{u_iv_i}$. By Lemma 
\ref{lem_ruv_5},  
$C_G$ satisfies (\ref{eqn_planegraphs})
and so Lemma \ref{lem_ordering2} 
and Corollary \ref{cor_totalorder}
apply to $C_G$.

Now for each $i$, let $R^i_1,\dots, R^i_{r_i}$
be the faces of $Q_i$, labelled so that $R^i_1$ 
is the unbounded face and so that $R^i_j$ is adjacent to $R^i_{j+1}$ for $j = 1,\dots, r_i-1$. 
We choose a set of vertices $Z_i \subset V(B)\backslash\{u_i,v_i\}$ as follows. Start with $Z_i = \emptyset$. Now suppose that 
$t$ is the smallest integer such that $5 \leq t \leq r_i -1$ and $R^i_t$ does not contain any vertex in $Z_i$. Let $Q_i = Q_{s_1} \succ Q_{s_2}  \succ \dots \succ Q_{s_k}$ be a chain in $C_G$ of maximal length such that $Q_{s_2}$ is contained in $R^i_t$, and for $m \geq 2$, $Q_{s_{m+1}}$ is contained in $R^{s_{m}}_3$. Moreover we choose so that, for $m \geq 1$, $Q_{s_{m+1}}$ is a maximal element (with respect to $\prec$) in the set $\{Q_l \in C_G: Q_l \prec Q_{s_m} \}$. So for $m \geq 1$ there is no element of $C_G$ strictly between $Q_{s_m}$ and $Q_{s_{m+1}}$.  

Since the chain above has maximal length, it follows that
\begin{equation}
    \label{eqn_noQinside}
    \text{$R^{s_{k}}_3$, respectively $R^i_t$, does not contain any $Q_j \in C_G$ if $k \geq 2$, respectively if $k=1$. }
\end{equation}

By Lemma \ref{lem_face_with_no_braced_vertices}, $R^{s_{k}}_3$, or $R^i_t$ if $k =1$, contains some vertex  $z \in V(B)\setminus \{u_{s_k},v_{s_k}\}$. In particular, since $z$ does not lie in the unbounded face of $Q_{s_{k}}$ it follows that $z \not\in \{u_i,v_i\}$. So $z \in V(B) \setminus \{u_i,v_i\}$ and $z$ lies in $R^i_t$. We add $z$ to the set $Z_i$. 

We continue choosing elements in this way until each of the faces $R^i_5,\dots,R^i_{r_i-1}$ contains at least one element of $Z_i$. Since no vertex in $V(P)\setminus \{u_i,v_i\}$ is contained in more than two of $R^i_5,\dots, R^1_{r_i-1}$ it follows that $|Z_i| \geq \frac{r_i-5}2$. Also since $Z_i \subset R^i_5 \cup \dots \cup R^i_{r_i-1}$ it follows that
\begin{equation}
\label{eqn_Zbounded}
\text{for $1 \leq i \leq c$ no element of $Z_i$ lies in $R^i_1$ or in $R^i_3$.}
\end{equation}

\begin{lem}
\label{lem_disjointZ}
$Z_i \cap Z_j = \emptyset$ for all $i \neq j$.
\end{lem}
\begin{proof}
Suppose that $z \in Z_i \cap Z_j$. 
By Lemma \ref{lem_ordering2} and (\ref{eqn_Zbounded}) it follows that $Q_i$ and $Q_j$ are comparable with respect to $\preceq$. Without loss of generality suppose that $Q_j \prec Q_i$. Now consider the sequence $Q_{s_1} \succ \dots \succ Q_{s_k}$ that is constructed during the selection of $z$ for $Z_i$. Since $z$ does not lie in the unbounded face of any $Q_{s_m}$, we see that for $m = 1,\dots,k$, $Q_{s_m} \in C^z_G$ (as defined in Corollary \ref{cor_totalorder}). Moreover, using (\ref{eqn_noQinside}), and since there is no element of $C_G$ strictly between $Q_{s_m} $ and $Q_{s_{m+1}}$ it follows that $\{ Q_{s_1},\ldots,Q_{s_k}\}  = \{Q_l \in C^z_G: Q_l \preceq Q_i\}$. In particular, since $Q_j \in C^z_G$ and $Q_j \prec Q_i$, we have  $Q_j = Q_{s_m}$ for some $ m \geq 2$. But this implies that $z \in R^j_3$ and, since $z \in Z_j$, this  contradicts (\ref{eqn_Zbounded}).
\end{proof}

\begin{proof}[Proof of Theorem \ref{thm_main_bound}.]
Suppose that $r_1,\dots,r_c$ and $Z_1,\dots,Z_c$ are as in the discussion above.

By Lemma \ref{lem_vertices} we have 
\begin{equation*}
    |V(P)| \leq 2b + \sum_{uv \in B}r_{uv}
\end{equation*}
Now, using $|Z_i| \geq (r_i-5)/2$, we have 
\begin{eqnarray*} 
\sum_{uv \in B} r_{uv} &=& \sum_{\{uv:r_{uv} \leq 5\}}r_{uv} +\sum_{i = 1}^c r_i  \\
&\leq & 5(b-c) +\sum_{i=1}^c (2|Z_i| +5) \\
&=& 5b + 2 \sum_{i=1}^c |Z_i|
\end{eqnarray*}
By Lemma \ref{lem_disjointZ}, $\sum_{i=1}^c |Z_i| = |\cup_{i=1}^c Z_i|$. 
Now suppose that $Q_l$ is maximal with respect to $\preceq$ in $C_G$ (there is at least one such $l$). Observe $u_l,v_l \not\in \cup_{i=1}^c Z_i$ since $u_l,v_l \in R^i_1$ for $i = 1,\dots,c$. 
Thus $|\cup_{i=1}^c Z_i| \leq |V(B)| - 2 \leq 2b-2$ 
and
$\sum_{uv \in B} r_{uv} \leq 9b-4$ as required.
\end{proof}

\section{Unibraced triangulations}
\label{s:unibrace}
A unibraced (i.e. \( b=1 \)) triangulation must have at least five vertices.
Up to homeomorphism of the sphere there is a unique triangulation
with five vertices. It follows immediately that 
up to homeomorphisms, there is exactly one unibraced triangulation
with five vertices. Observe that in this unibraced triangulation the 
three vertices that are not in the brace must span a nonfacial triangle
of \( P \). The following is implicit in \cite{JordanTanigawa}. Also see 
\cite{CJ} for related results.

\begin{thm}
    \label{thm_unibraced}
    Every unibraced triangulation with at least six vertices has 
    a contractible edge. Equivalently, the unibraced triangulation 
    with five vertices is the unique irreducible unibraced triangulation.
\end{thm}

\begin{proof}
    Suppose that \( G \) is an irreducible unibraced triangulation
    with brace \( uv \). Since \( G \) has at least five vertices, it follows from 
		Lemma \ref{lem_vertices} that $r_{uv}\geq 3$.
    By Lemma \ref{lem_face_with_no_braced_vertices}, \( r_{uv} \leq 3 \). 
		Thus, $r_{uv}=3$.
    The conclusion now follows from Lemma \ref{lem_vertices}.
\end{proof}

With a little more effort 
we can strengthen Theorem \ref{thm_unibraced} as follows.

\begin{lem}
    \label{lem_cont_away_from_triangle}
    Suppose that \( G  = (P,B) \) is a unibraced triangulation with 
    at least six vertices. Let \( T \) be a face of \( P \). There is
    some edge of \( P \), not in \( T \), that is contractible 
    in \( G \).
\end{lem}

\begin{proof}
    Let \( u_i, i = 1,2,3 \) be the vertices of \( T \). By Lemma 
    \ref{lem_vertex_incident_good_edge} there are edges \( e_i, i = 1,2,3 \)
    such that \( e_i \) is incident with $u_i$, $e_i$ is not an edge of \( T \) 
		and \( e_i \) is contractible in \( P \). If any \( e_i \) is not adjacent 
		to the brace then we are done. So we may as well assume that each
    \( e_i \) is adjacent to the brace for \( i = 1,2,3 \).

    It follows, in particular, that  \( e_1,e_2,e_3 \) cannot be 
    pairwise non-incident, so there are two cases to consider. 
    Either (a) \( e_1 \) and 
    \( e_2 \) share a vertex and \( e_3 \) is non-adjacent to 
    \( e_1 \) and \( e_2 \), or, (b) \( e_1,e_2,e_3 \) have 
    a common vertex.

    In case (b), let \( p  \) be the common vertex of \( e_1,e_2,e_3 \). 
    Clearly \( p \) must be one vertex of the brace and it follows 
    from the planarity of \( P \)  
    that the other vertex of the brace must lie in the interior of 
    the face \( T \) (since it is also adjacent in \( P \)
    to \( u_1,u_2,u_3 \)). 
    This contradicts the fact that \( T \) is a face of \( P \) and 
    so has no vertices in its interior, by definition (see Figure~\ref{fig:unibracedallinp}).

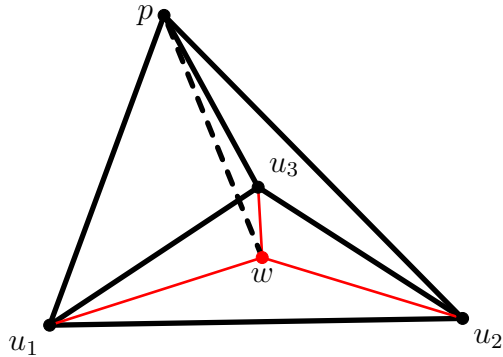
\begin{figure}[htp]
        \centering

        \definecolor{ududff}{rgb}{0,0,0}
    \begin{tikzpicture}[scale=0.9,line cap=round,line join=round,>=triangle 45,x=1cm,y=1cm]
    \draw [line width=2pt] (-10.31,3.44)-- (-13.39,1.4);
    \draw [line width=2pt] (-13.39,1.4)-- (-7.29,1.5);
    \draw [line width=2pt] (-7.29,1.5)-- (-10.31,3.44);
    \draw [line width=1pt,red] (-10.25,2.4)-- (-10.31,3.44);
		
    \draw [line width=1pt,red] (-10.25,2.4)-- (-7.29,1.5);
    \draw [line width=1pt,red] (-10.25,2.4)-- (-13.39,1.4);

    \draw [line width=2pt,dash pattern=on 5pt off 5pt] (-10.25,2.4)-- (-11.7,5.98);

    \draw [line width=2pt] (-10.31,3.44)-- (-11.7,5.98);
    \draw [line width=2pt] (-13.39,1.4)-- (-11.7,5.98);
    \draw [line width=2pt] (-7.29,1.5)-- (-11.7,5.98);

    \begin{scriptsize}
    \draw [fill=ududff] (-10.31,3.44) circle (2.5pt);
    \draw [fill=ududff] (-13.39,1.4) circle (2.5pt);
    \draw [fill=ududff] (-7.29,1.5) circle (2.5pt);
    \draw [red, fill] (-10.25,2.4) circle (2.5pt);
    \draw [fill=ududff] (-11.7,5.98) circle (2.5pt);
    \end{scriptsize}
		
		\node [below left] at (-13.39,1.4) {$u_1$};
    \node [below right] at (-7.29,1.5) {$u_2$};
    \node [above right] at (-10.31,3.44) {$u_3$};
    \node [below] at (-10.25,2.4) {$w$};
    \node [left] at (-11.7,5.98) {$p$};
    \end{tikzpicture}

    \caption{If $u_1p, u_2p,u_3p$ are uncontractible edges, then $V(B)=\{p,w\}$ where $w$ lies in the interior of the face $T$.}
        \label{fig:unibracedallinp}
\end{figure}

    Thus, only case (a) remains. Let \( p \) be the common vertex of 
    \( e_1 \) and \( e_2 \) and let \( q \) be the vertex of \( e_3 \) 
    that is different from \( u_3 \). It is clear that the brace 
    must be either \( pq \) or \( pu_3 \). If \( pq \) is the brace
    then \( u_1q, u_2q ,pu_3 \) must all be edges of \( P \). It is 
    not hard to see that this contradicts the planarity of \( P \).

    Finally suppose that \( pu_3 \) is the brace. 
    Note that \( u_1,u_2 \in X_{pu_3} \). 
    Let \( v \) 
    be a vertex of \( G \) that is not in \( \{u_1,u_2,u_3,p\} \). 
    By Lemma \ref{lem_good_edge_not_in_face} there are two 
    edges of \( P \) incident to \( v \) that are contractible in \( P \).
    If any such edge is not adjacent to the brace \( pu_3 \) then 
    we are done. Thus we may assume that \( V(G) = \{p,u_3\} \cup X_{pu_3} \)
    and that \( G \) is in fact a braced $n$-gonal bipyramid  
    where the brace joins the two poles. 
    Since \( G \) has at least 
    six vertices, all of the equatorial edges are contractible in 
    \( G \) and so
    the conclusion of the lemma is true in 
    this case also.
\end{proof}

\section{Doubly braced triangulations}
\label{sec_b2}

In this section we will show that in the case \( b=2 \)
there are  five irreducibles. These are shown in Figures 
\ref{fig_doubly_braced_oct}, \ref{fig_hex_disjoint}, 
\ref{fig_hex_adjacent}, \ref{fig_ireed_seven} and 
\ref{fig_ireed_seven_non}. The braces are indicated
with dotted lines.

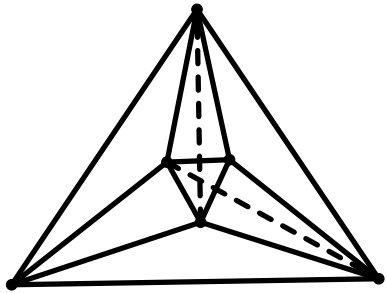
\begin{figure}[htp]
        \centering

        \definecolor{ududff}{rgb}{0,0,0}
    \begin{tikzpicture}[scale=0.8,line cap=round,line join=round,>=triangle 45,x=1cm,y=1cm]
    \draw [line width=2pt] (-10.31,5.98)-- (-13.39,1.4);
    \draw [line width=2pt] (-13.39,1.4)-- (-7.29,1.5);
    \draw [line width=2pt] (-7.29,1.5)-- (-10.31,5.98);
    \draw [line width=2pt] (-10.81,3.44)-- (-9.77,3.48);
    \draw [line width=2pt] (-9.77,3.48)-- (-10.25,2.44);
    \draw [line width=2pt] (-10.25,2.44)-- (-10.81,3.44);
    \draw [line width=2pt] (-10.81,3.44)-- (-10.31,5.98);
    \draw [line width=2pt] (-10.81,3.44)-- (-13.39,1.4);
    \draw [line width=2pt] (-9.77,3.48)-- (-10.31,5.98);
    \draw [line width=2pt] (-9.77,3.48)-- (-7.29,1.5);
    \draw [line width=2pt] (-10.25,2.44)-- (-7.29,1.5);
    \draw [line width=2pt] (-10.25,2.44)-- (-13.39,1.4);
    \draw [line width=2pt,dash pattern=on 5pt off 5pt] (-10.81,3.44)-- (-7.29,1.5);
    \draw [line width=2pt,dash pattern=on 5pt off 5pt] (-10.25,2.44)-- (-10.31,5.98);
    \begin{scriptsize}
    \draw [fill=ududff] (-10.31,5.98) circle (2.5pt);
    \draw [fill=ududff] (-13.39,1.4) circle (2.5pt);
    \draw [fill=ududff] (-7.29,1.5) circle (2.5pt);
    \draw [fill=ududff] (-10.81,3.44) circle (2.5pt);
    \draw [fill=ududff] (-9.77,3.48) circle (2.5pt);
    \draw [fill=ududff] (-10.25,2.44) circle (2.5pt);
    \end{scriptsize}
    \end{tikzpicture}

    \caption{The doubly braced octahedron.}
        \label{fig_doubly_braced_oct}
\end{figure}

\begin{figure}[htp]
    \definecolor{ududff}{rgb}{0,0,0}
    \begin{tikzpicture}[scale=0.8,line cap=round,line join=round,>=triangle 45,x=1cm,y=1cm]
    \draw [line width=2pt] (-3.04,0.08)-- (0.12,5.56);
    \draw [line width=2pt] (0.12,5.56)-- (3.54,0.24);
    \draw [line width=2pt] (3.54,0.24)-- (-3.04,0.08);
    \draw [line width=2pt] (-3.04,0.08)-- (0.06,0.94);
    \draw [line width=2pt] (0.06,0.94)-- (3.54,0.24);
    \draw [line width=2pt] (1.16,2.94)-- (-3.04,0.08);
    \draw [line width=2pt] (1.16,2.94)-- (0.06,0.94);
    \draw [line width=2pt] (1.16,2.94)-- (3.54,0.24);
    \draw [line width=2pt] (0.12,5.56)-- (1.16,2.94);
    \draw [line width=2pt,dash pattern=on 5pt off 5pt] (0.12,5.56)-- (0.06,0.94);
    \draw [line width=2pt] (0.06,0.94)-- (0.86,1.58);
    \draw [line width=2pt] (0.86,1.58)-- (1.16,2.94);
    \draw [line width=2pt] (0.86,1.58)-- (3.54,0.24);
    \draw [line width=2pt,dash pattern=on 5pt off 5pt] (0.86,1.58)-- (-3.04,0.08);
    \begin{scriptsize}
    \draw [fill=ududff] (-3.04,0.08) circle (2.5pt);
    \draw [fill=ududff] (0.12,5.56) circle (2.5pt);
    \draw [fill=ududff] (3.54,0.24) circle (2.5pt);
    \draw [fill=ududff] (0.06,0.94) circle (2.5pt);
    \draw [fill=ududff] (1.16,2.94) circle (2.5pt);
    \draw [fill=ududff] (0.86,1.58) circle (2.5pt);
    \end{scriptsize}
    \end{tikzpicture}
    \caption{Capped hexahedron with disjoint braces}
    \label{fig_hex_disjoint}
\end{figure}
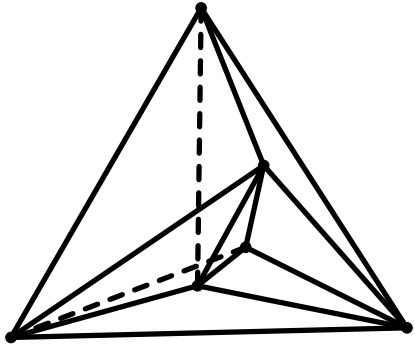

\begin{figure}[htp]
    \definecolor{ududff}{rgb}{0,0,0}
    \begin{tikzpicture}[scale=0.8,line cap=round,line join=round,>=triangle 45,x=1cm,y=1cm]
    \draw [line width=2pt] (-3.04,0.08)-- (0.12,5.56);
    \draw [line width=2pt] (0.12,5.56)-- (3.54,0.24);
    \draw [line width=2pt] (3.54,0.24)-- (-3.04,0.08);
    \draw [line width=2pt] (-3.04,0.08)-- (0.06,0.94);
    \draw [line width=2pt] (0.06,0.94)-- (3.54,0.24);
    \draw [line width=2pt] (1.16,2.94)-- (-3.04,0.08);
    \draw [line width=2pt] (1.16,2.94)-- (0.06,0.94);
    \draw [line width=2pt] (1.16,2.94)-- (3.54,0.24);
    \draw [line width=2pt] (0.12,5.56)-- (1.16,2.94);
    \draw [line width=2pt,dash pattern=on 5pt off 5pt] (0.12,5.56)-- (0.06,0.94);
    \draw [line width=2pt] (0.06,0.94)-- (0.86,1.58);
    \draw [line width=2pt] (0.86,1.58)-- (1.16,2.94);
    \draw [line width=2pt] (0.86,1.58)-- (3.54,0.24);
    \draw [line width=2pt,dash pattern=on 5pt off 5pt] (0.86,1.58)-- (0.12,5.56);
    \begin{scriptsize}
    \draw [fill=ududff] (-3.04,0.08) circle (2.5pt);
    \draw [fill=ududff] (0.12,5.56) circle (2.5pt);
    \draw [fill=ududff] (3.54,0.24) circle (2.5pt);
    \draw [fill=ududff] (0.06,0.94) circle (2.5pt);
    \draw [fill=ududff] (1.16,2.94) circle (2.5pt);
    \draw [fill=ududff] (0.86,1.58) circle (2.5pt);
    \end{scriptsize}
    \end{tikzpicture}
    \caption{Capped hexahedron with adjacent braces}
    \label{fig_hex_adjacent}
\end{figure}
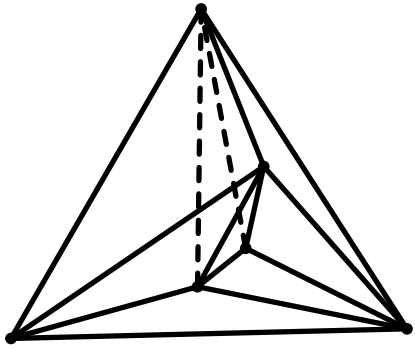

\begin{figure}[htp]
    \definecolor{ududff}{rgb}{0,0,0}
    \begin{tikzpicture}[scale=0.7, line cap=round,line join=round,>=triangle 45,x=1cm,y=1cm]
    \draw [line width=2pt] (-3.32,-0.6)-- (5.46,-0.88);
    \draw [line width=2pt] (5.46,-0.88)-- (1.92,5.98);
    \draw [line width=2pt] (1.92,5.98)-- (-3.32,-0.6);
    \draw [line width=2pt] (1.3,1.28)-- (1.12,2.6);
    \draw [line width=2pt] (1.12,2.6)-- (1.1,4.08);
    \draw [line width=2pt] (1.1,4.08)-- (1.92,5.98);
    \draw [line width=2pt] (1.3,1.28)-- (-3.32,-0.6);
    \draw [line width=2pt] (1.12,2.6)-- (-3.32,-0.6);
    \draw [line width=2pt] (1.1,4.08)-- (-3.32,-0.6);
    \draw [line width=2pt] (1.1,4.08)-- (5.46,-0.88);
    \draw [line width=2pt] (1.12,2.6)-- (5.46,-0.88);
    \draw [line width=2pt] (1.3,1.28)-- (5.46,-0.88);
    \draw [line width=2pt] (1.3,1.28)-- (2.5,-0.06);
    \draw [line width=2pt] (2.5,-0.06)-- (-3.32,-0.6);
    \draw [line width=2pt] (2.5,-0.06)-- (5.46,-0.88);
    \draw [line width=2pt,dash pattern=on 5pt off 5pt] (1.12,2.6)-- (1.92,5.98);
    \draw [line width=2pt,dash pattern=on 5pt off 5pt] (1.12,2.6)-- (2.5,-0.06);
    \begin{scriptsize}
    \draw [fill=ududff] (-3.32,-0.6) circle (2.5pt);
    \draw [fill=ududff] (5.46,-0.88) circle (2.5pt);
    \draw [fill=ududff] (1.92,5.98) circle (2.5pt);
    \draw [fill=ududff] (1.3,1.28) circle (2.5pt);
    \draw [fill=ududff] (1.12,2.6) circle (2.5pt);
    \draw [fill=ududff] (1.1,4.08) circle (2.5pt);
    \draw [fill=ududff] (2.5,-0.06) circle (2.5pt);
    \end{scriptsize}
    \end{tikzpicture}
    \caption{Irreducible with seven vertices and adjacent braces}
    \label{fig_ireed_seven}
\end{figure}
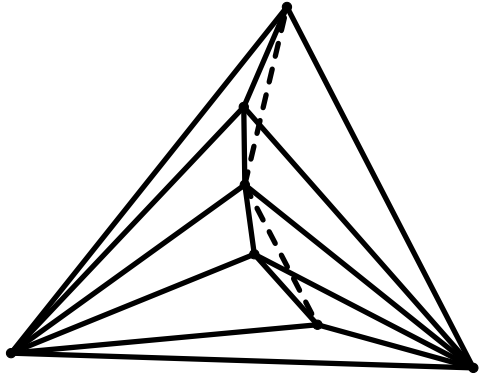

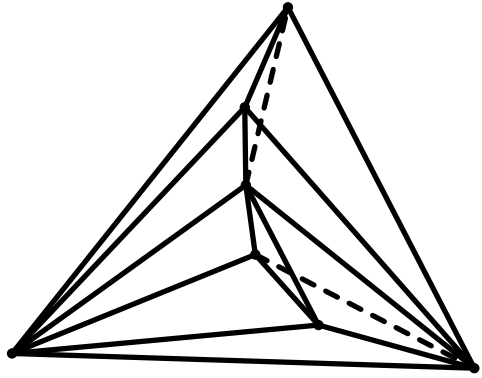
\begin{figure}[htp]
    \definecolor{ududff}{rgb}{0,0,0}
    \begin{tikzpicture}[scale=0.7,line cap=round,line join=round,>=triangle 45,x=1cm,y=1cm]
    \draw [line width=2pt] (-3.32,-0.6)-- (5.46,-0.88);
    \draw [line width=2pt] (5.46,-0.88)-- (1.92,5.98);
    \draw [line width=2pt] (1.92,5.98)-- (-3.32,-0.6);
    \draw [line width=2pt] (1.3,1.28)-- (1.12,2.6);
    \draw [line width=2pt] (1.12,2.6)-- (1.1,4.08);
    \draw [line width=2pt] (1.1,4.08)-- (1.92,5.98);
    \draw [line width=2pt] (1.3,1.28)-- (-3.32,-0.6);
    \draw [line width=2pt] (1.12,2.6)-- (-3.32,-0.6);
    \draw [line width=2pt] (1.1,4.08)-- (-3.32,-0.6);
    \draw [line width=2pt] (1.1,4.08)-- (5.46,-0.88);
    \draw [line width=2pt] (1.12,2.6)-- (5.46,-0.88);
    \draw [line width=2pt,dash pattern=on 5pt off 5pt] (1.3,1.28)-- (5.46,-0.88);
    \draw [line width=2pt] (1.3,1.28)-- (2.5,-0.06);
    \draw [line width=2pt] (2.5,-0.06)-- (-3.32,-0.6);
    \draw [line width=2pt] (2.5,-0.06)-- (5.46,-0.88);
    \draw [line width=2pt,dash pattern=on 5pt off 5pt] (1.12,2.6)-- (1.92,5.98);
    \draw [line width=2pt] (1.12,2.6)-- (2.5,-0.06);
    \begin{scriptsize}
    \draw [fill=ududff] (-3.32,-0.6) circle (2.5pt);
    \draw [fill=ududff] (5.46,-0.88) circle (2.5pt);
    \draw [fill=ududff] (1.92,5.98) circle (2.5pt);
    \draw [fill=ududff] (1.3,1.28) circle (2.5pt);
    \draw [fill=ududff] (1.12,2.6) circle (2.5pt);
    \draw [fill=ududff] (1.1,4.08) circle (2.5pt);
    \draw [fill=ududff] (2.5,-0.06) circle (2.5pt);
    \end{scriptsize}
    \end{tikzpicture}
    \caption{Irreducible with seven vertices and non-adjacent braces}
    \label{fig_ireed_seven_non}
\end{figure}

First we observe that there are precisely two distinct triangulations of the
sphere with six vertices. They are the octahedron and the capped hexahedron. 
It is not hard to deduce that there are three distinct braced triangulations with 
six vertices and that each of these is irreducible since any braced triangulation 
with two braces must have at least six vertices. 
Observe that each of these three irreducibles has underlying graph isomorphic 
to \( K_6 \) with one edge removed (which is, of course, the unique six vertex graph
with 14 edges).
The two seven vertex irreducibles also have isomorphic underlying graphs.
In those cases the graph is isomorphic to that obtained by gluing 
two copies of \( K_5 \) together along a \( K_3 \).

\begin{thm}
    \label{thm_irred_seven}
    Any irreducible braced triangulation with 
    two braces is isomorphic 
    to one of the examples shown in Figures \ref{fig_doubly_braced_oct},
    \ref{fig_hex_disjoint}, \ref{fig_hex_adjacent}, \ref{fig_ireed_seven}
    or \ref{fig_ireed_seven_non}.
\end{thm}

Clearly it suffices to show that any irreducible with at least 
seven vertices is isomorphic to one of the examples shown in 
Figures \ref{fig_ireed_seven} or \ref{fig_ireed_seven_non}.
The rest of this section is devoted to proving that.
We observe that by Theorem 
\ref{thm_quad_bound}, any irreducible doubly braced triangulation 
has at most 12 vertices. Thus, in principle at least, we have a finite list
of candidates among which we can search for irreducibles. However 
since there are a large number of doubly braced triangulations with
at most 12 vertices we find it desirable instead to 
narrow the search space by improving the general bound of Theorem 
\ref{thm_quad_bound} in the case \( b=2 \).

Suppose that \( G = (P,B) \) is an irreducible braced 
triangulation with \( B = \{uv,wx\} \). 
As above, let \( Q_{uv} \) be the bipartite sphere graph induced by 
\( K_{\{u,v\},X_{uv}} \) and, when $r_{uv}\geq 2$, let \( R_1,\cdots,R_{r_{uv}} \) be the faces of \( Q_{uv} \). Furthermore \( X_{uv} = \{y_1,y_2,\cdots,y_{r_{uv}}\} \)
where \( u,y_{i},v,y_{i+1} \) are the boundary vertices 
of \( R_{i} \) for \( i = 1,\cdots,r_{uv} \) (adopting the convention 
that \( y_{r_{uv}+1} = y_1 \)).

\begin{lem}
    \label{lem_octa}
		Let \( G = (P,B) \) be an irreducible braced 
triangulation with \( B = \{uv,wx\} \).
    Suppose that \( \max\{r_{uv},r_{wx}\} \geq 4 \). Then \( G \) is the 
    doubly braced octahedron (Figure \ref{fig_doubly_braced_oct}).
\end{lem}

\begin{proof}
    Without loss of generality, suppose that \( r_{uv} \geq 4 \).
    By Lemma \ref{lem_face_with_no_braced_vertices},
    we see that each \( R_i \) contains some element of 
    \( V(B) - \{u,v\} = \{w,x\} \). Note that \( w \) (and \( x \)) 
    can belong to at most two faces of $Q_{uv}$. It follows that \(  w,x \in X_{uv} \) and $r_{uv}=4$.
    Now since \( w,x \) do not belong to any common face of \( Q_{uv} \),
    it is clear that \( X_{wx} \subset \{u,v\} \cup X_{uv} \).
    Using Lemma \ref{lem_vertices} we see that \( |V(G)| = 6\) and the 
    conclusion follows easily since the doubly braced octahedron
    is the only six vertex irreducible that satisfies \( \max \{r_{uv},
    r_{wx}\} \geq 4\).
\end{proof}

Thus we may assume from now on that \( \max\{r_{wx},r_{uv}\} \leq 3 \).

\begin{lem}
    \label{lem_interior}
		Let \( G = (P,B) \) be an irreducible braced 
triangulation with \( B = \{uv,wx\} \).
    Suppose that \( \max\{r_{wx},r_{uv}\} \leq 3 \). 
    Then at most one of \( w,x \) is in \( \{u,v\} \cup X_{uv} \).
\end{lem}

\begin{proof}
    Suppose that \( w,x \in \{u,v\}\cup X_{uv} \). Since 
    \( wx \) is not an edge of \( P \) and \( wx \neq uv \), it follows 
    that \( \{w,x\} \subset X_{uv} \).
    Since \( r_{uv} \leq 3 \), we may assume, 
    without loss of generality, that \( \{w,x\} = \{y_1,y_2\} \).
    Now since \( P \) is a triangulation and since 
    neither \( uv \) nor \( wx \) can be edges of \( P \) it follows 
    that there are some vertices in \( \mathring R_1 \) 
    (\( \mathring R_1 \) denotes the interior of \( R_1 \)). By Lemma
    \ref{lem_vertices}, all such vertices must be in \( X_{wx} \).
    However, in this situation \( u,v \in X_{wx} \) and since 
    \( r_{wx} \leq 3 \), it follows that there is exactly one vertex, 
    \( z \),
    in \( \mathring R_1 \). Since \( P \) is a triangulation
    not containing \( uv, wx \) it must be that 
    \( zx,zw,zu,zv \) are all edges of \( P \).
    Thus
    \( z \in X_{uv} \)
    which contradicts the fact that \( z \) is in \( \mathring R_1 \).
\end{proof}

\begin{lem}
    \label{lem_max_at_least3}
    If \( G \) is an irreducible doubly braced triangulation with braces
    \( uv \) and \( wx \) then
    \( \max\{r_{uv},r_{wx}\} \geq 3 \).
\end{lem}

\begin{figure}[ht]
    \centering

    \definecolor{ududff}{rgb}{0,0,0}
\begin{tikzpicture}[scale=0.7,line cap=round,line join=round,>=triangle 45,x=1cm,y=1cm]
\draw [line width=2pt] (-7.58,4.33)-- (-9.92,0.67);
\draw [line width=2pt] (-7.58,4.33)-- (-5.98,0.67);
\draw [line width=2pt] (-9.92,0.67)-- (-8.08,-2.51);
\draw [line width=2pt] (-8.08,-2.51)-- (-5.98,0.67);
\draw [line width=2pt] (-9.92,0.67)-- (-7.9,0.25);
\draw [line width=2pt] (-7.9,0.25)-- (-5.98,0.67);
\draw [shift={(-7.95,2.6056310679611654)},line width=2pt]  plot[domain=-0.7765985496250778:3.918191203214871,variable=\t]({1*2.7618051399866146*cos(\t r)+0*2.7618051399866146*sin(\t r)},{0*2.7618051399866146*cos(\t r)+1*2.7618051399866146*sin(\t r)});
\begin{scriptsize}
\draw [fill=ududff] (-7.58,4.33) circle (2.5pt);
\draw[color=ududff] (-6.42,4.66) node {$u$};
\draw [fill=ududff] (-9.92,0.67) circle (2.5pt);
\draw[color=ududff] (-10.3,0.74) node {$y_1$};
\draw [fill=ududff] (-5.98,0.67) circle (2.5pt);
\draw[color=ududff] (-5.62,0.64) node {$y_2$};
\draw [fill=ududff] (-8.08,-2.51) circle (2.5pt);
\draw[color=ududff] (-7.92,-2.76) node {$v$};
\draw [fill=ududff] (-7.9,0.25) circle (2.5pt);
\draw[color=ududff] (-7.79,0.58) node {$z$};
\draw (-7.79,2.28) node {$R_1^+$};
\draw (-7.99,-1.08) node {$R_1^-$};
\end{scriptsize}
\end{tikzpicture}
    \caption{A sketch for the proof of Lemma \ref{lem_max_at_least3}.}
    \label{fig_at_least_three}
\end{figure}
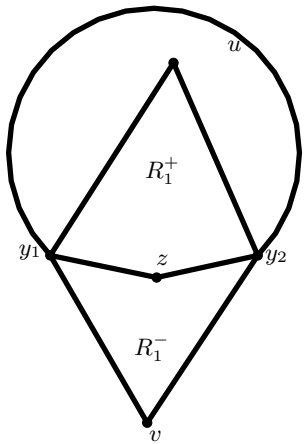

\begin{proof}
    For a contradiction 
    assume that \( r_{wx} \leq r_{uv} \leq 2 \). 
    
    Suppose that \( r_{wx} + r_{uv}  \leq 2 \). By Lemma \ref{lem_vertices}
    we see that \(|V(G)| \leq 6  \). However it is easy to see that there are only
    three irreducible doubly braced triangulations with at most six vertices
    (these are shown in Figures \ref{fig_doubly_braced_oct},  
    \ref{fig_hex_disjoint} and \ref{fig_hex_adjacent}) and none of these 
    satisfy \( r_{wx} +r_{uv} \leq 2 \).

    Now suppose that \( r_{wx} \leq 2 \) and \( r_{uv} = 2 \). By Lemma
    \ref{lem_interior} we may
    suppose that \( w \in \mathring R_1 \). If \( x \not \in R_1 \) then,
		by planarity we have $X_{wx}\subseteq\{u,v,y_1,y_2\}$.
		Thus, by Lemma \ref{lem_vertices}, we have \(|V(G)| \leq 6  \).
		As in the previous paragraph, this is not possible since 
		in each of the Figures \ref{fig_doubly_braced_oct},  
    \ref{fig_hex_disjoint} and \ref{fig_hex_adjacent} we have
		\( \max\{r_{uv},r_{wx}\} \geq 3 \).
		So we can assume $x\in R_1$. 
		Now since $w\in R_1^\circ$, it follows that $X_{wx}\subseteq R_1$.
        
    By Lemma \ref{lem_vertices}, there are no vertices in \( \mathring R_2 \). It follows 
    that \( y_1y_2 \) is an edge of \( P \) that is contained in \( R_2 \).
    If this edge is part of a braced triangle then wlog \( y_1 \in 
    \{w,x\}\) and \( y_2 \in X_{wx} \). Thus, in this case, 
    \( |V(G)| = |V(B) \cup X_{uv} \cup X_{wx}| \leq 6 \) and  we know that 
    no six vertex irreducible satisfies 
    \( \max\{r_{uv},r_{wx}\} \leq 2 \).
    
    So \( y_1y_2 \) is not part of any braced triangle. It must 
    therefore be part of a non-facial triangle. Thus there is a vertex 
    \( z \in \mathring R_1 \) and edges \( y_1z \) and \( zy_2 \). See 
    Figure \ref{fig_at_least_three} for an illustration. Note that 
    \( R_1 \) splits into two closed regions
    \( R_1^+ \) and \( R_1^- \) whose intersection is the path
    \( y_1,z,y_2  \) as shown in Figure \ref{fig_at_least_three}.
    
    Now, we claim that in fact \( w = z \). If not, then since \( w \in 
    \mathring R_1\) it follows that \( w \) is contained 
    in the interior of either \( R_1^+ \) or \( R_1^- \)
    (see Figure \ref{fig_at_least_three}), wlog say \( R_1^+ \). 
		Suppose $x$ lies in the interior of $R_1^-$. In this case, by Lemma \ref{lem_vertices}, the interior of $R_1^-$ contains no other vertices. There are now exactly two ways to	triangulate the region $R_1^-$.
		For each of these triangulations note that $xv$ is an edge of $P$ which is contractible in $G$, a contradiction. 
		Thus we may assume that $x$ does not lie in the interior of $R_1^-$.
    It follows that \( zv \in E(P) \). Note that
    the edge $zv$ cannot lie in a nonfacial triangle of \( P \).
		Also, since $z\notin X_{uv}$, the edge $zv$ cannot lie in a $3$-cycle containing the brace $uv$. Thus $zv$ must lie in a $3$-cycle containing the brace $wx$.
		It follows that \( x =v \) and that \( wz \in E(P) \). Since there
    can be no other vertices in \( R_1^+ \) it follows that 
    \( wy_1 \) and \( wy_2 \) are edges in \( P \), whence 
    \( r_{wx} \geq 3 \). Thus our assumption that \( w \neq z \)
    leads to a contradiction. 

    On the other hand, if \( x \in \mathring R_1 \), then we may assume without loss of generality that $x$ lies in the interior of $R_1^-$.
		In this case, there are no vertices in the interior of $R_1^+$. Moreover, the edge $uw$ lies in $P$ and is contractible in $G$, a contradiction. 
		Thus we may assume that  
    \( x \in \{ u,v\} \): wlog \( x = u \). 
    Now it is clear that \( X_{wx}=X_{uv} =\{y_1,y_2\} \). Thus, by Lemma \ref{lem_vertices} we have $|V(G)| = 5$, a contradiction. 
\end{proof}

\begin{lem}
    \label{lem_not_in_face_of_Q}
    Suppose that 
    \( r_{wx} \leq r_{uv}=3 \) and that 
    there is no face of \( Q_{uv} \) that contains both 
    \( w \) and \( x \). Then \( G \) is a doubly braced 
    capped hexahedron with disjoint braces (see Figure 
    \ref{fig_hex_disjoint}).
\end{lem}

\begin{proof}
    By Lemma \ref{lem_interior}, we 
    can assume that \( w \in \mathring{R}_1 \).
    So 
    \( x \not \in R_1 \) and since, by Lemma
    \ref{lem_vertices}, \( V(G) = \{u,v,w,x\}\cup X_{uv}\cup X_{wx} \)
    we see that \( w \) is the only vertex in \( \mathring R_1 \). 
    Thus \( N_P(w) \subset \{y_1,u,y_2,v\} \). Also \( |N_P(w)| \geq 3 \)
    since the min degree of any triangulation is at least three.
    Furthermore
    \( w \not\in X_{uv} \), so it follows that \( w \)
    is adjacent to both of \( y_1,y_2 \) and exactly one of \( u,v \), say 
    \( u \) wlog. 
    In a triangulation an edge incident to a vertex of degree 3
    cannot belong to a nonfacial triangle. Thus
    none of \( wy_1,  wy_2, wu \) are 
    in nonfacial triangles. 
		Also, since \( wv \) is not an edge of 
    \( G \), the edge \( wu \) is not in a triangle
    that contains the brace \( uv \). It follows that 
    the  edges \( wy_1,wy_2, wu\) must all belong to triangles 
    containing the brace \( wx \). Thus \( x \) is a common neighbour in 
    \( P \) of \( y_1,y_2 \) and \( u \) that is not in \( R_1 \).
		Since $r_{uv}=3$ we must have by planarity that 
		$X_{uv}=\{x,y_1,y_2\}$. 
    It now follows easily that \( G \) is the doubly braced 
    capped hexahedron with disjoint braces as claimed.
\end{proof}

\begin{lem}
    \label{lem_three_cycle}
    Suppose that \( r_{wx} \leq r_{uv} = 3 \). Then \( X_{uv} \) spans 
    a 3-cycle in \( P \).
\end{lem}

\begin{proof}
    If no face of \( Q_{uv} \) contains both \( w,x \) then by Lemma
    \ref{lem_not_in_face_of_Q}, \( G \) is isomorphic to the capped hexahedron  
    shown in Figure \ref{fig_hex_disjoint} and the conclusion is true.
    Using this and Lemma \ref{lem_interior} we may assume that 
    \( w \in \mathring R_1 \) and \( x
    \in R_1\). Now it is clear that there are no vertices in 
    \( \mathring R_2 \) or in \( \mathring R_3 \) since 
    such vertices would have to be in \( X_{wx} \), by Lemma 
    \ref{lem_vertices} and this contradicts \( w \in \mathring R_1 \). 
    Now, since \( uv \) is not an edge of \( P \), it follows that 
    \( y_2y_3 \) and \( y_3y_1 \) are both edges of \( P \). Since 
    \( y_3 \not\in X_{wx} \) we also conclude that both of these 
    edges must lie in nonfacial triangles of \( P \). Furthermore, 
    it follows that 
    \( N_{P}(y_3) = \{u,v,y_1,y_2\}\). Therefore \( y_1y_2 \) 
    must also be an edge of \( P \) as required.
\end{proof}

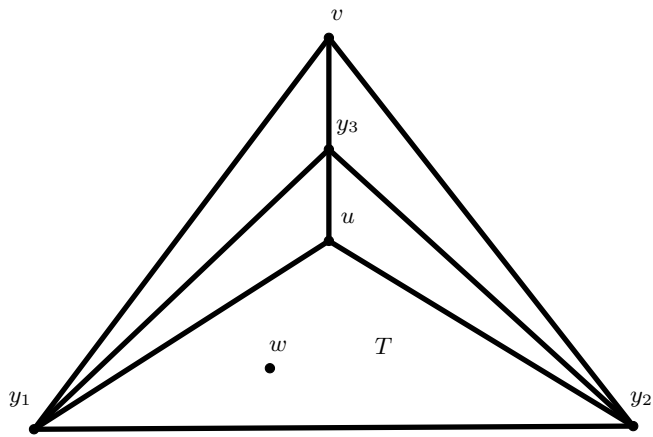
\begin{figure}[ht]
    \centering
    \definecolor{ududff}{rgb}{0,0,0}
\begin{tikzpicture}[scale=0.7,line cap=round,line join=round,>=triangle 45,x=1cm,y=1cm]
\draw [line width=2pt] (-8.39,5.72)-- (-13.99,-1.72);
\draw [line width=2pt] (-13.99,-1.72)-- (-2.61,-1.66);
\draw [line width=2pt] (-2.61,-1.66)-- (-8.39,5.72);
\draw [line width=2pt] (-8.39,1.86)-- (-13.99,-1.72);
\draw [line width=2pt] (-8.39,1.86)-- (-2.61,-1.66);
\draw [line width=2pt] (-8.39,3.6)-- (-13.99,-1.72);
\draw [line width=2pt] (-8.39,3.6)-- (-2.61,-1.66);
\draw [line width=2pt] (-8.39,3.6)-- (-8.39,1.86);
\draw [line width=2pt] (-8.39,3.6)-- (-8.39,5.72);
\begin{scriptsize}
\draw [fill=ududff] (-8.39,5.72) circle (2.5pt);
\draw[color=ududff] (-8.23,6.15) node {$v$};
\draw [fill=ududff] (-13.99,-1.72) circle (2.5pt);
\draw[color=ududff] (-14.25,-1.15) node {$y_1$};
\draw [fill=ududff] (-2.61,-1.66) circle (2.5pt);
\draw[color=ududff] (-2.45,-1.15) node {$y_2$};
\draw [fill=ududff] (-8.39,1.86) circle (2.5pt);
\draw[color=ududff] (-8.03,2.29) node {$u$};
\draw [fill=ududff] (-8.39,3.6) circle (2.5pt);
\draw[color=ududff] (-8.03,4.03) node {$y_3$};
\draw [fill=ududff] (-9.51,-0.56) circle (2.5pt);
\draw[color=ududff] (-9.35,-0.13) node {$w$};
\draw[color=ududff] (-7.35,-0.13) node {$T$};
\end{scriptsize}
\end{tikzpicture}
\caption{The case \( r_{wx} \leq r_{uv} = 3 \). Here \( R_1 \) is split into two triangular regions by the edge \( y_1y_2 \). One is the unbounded region (in this plane embedding) and one is the region containing \( w \).}
    \label{fig_max3}
\end{figure}

\begin{proof}[Proof of Theorem \ref{thm_irred_seven}]
    If \( \max\{r_{uv},r_{wx}\} \geq 4 \) then, by Lemma \ref{lem_octa}, \( G \) is isomorphic to
    the example of Figure \ref{fig_doubly_braced_oct}. 
    So we assume that \( r_{wx} \leq r_{uv} \leq 3 \). 
    By Lemma \ref{lem_max_at_least3} we have \( r_{uv} = 3 \).

Now using Lemmas \ref{lem_interior}
and \ref{lem_three_cycle} we have the situation illustrated in Figure
\ref{fig_max3}. Note that any vertices that are not in \( X_{uv}\cup \{u,v,w,x\} \) 
must  lie in the interior of the triangular region labelled \( T \) in 
Figure \ref{fig_max3}, since any such vertex must be in \( X_{wx} \).

Now suppose that \( x \) does not lie in the closed region \( T \).
If $x\notin R_1$ then, by Lemma \ref{lem_not_in_face_of_Q}, $G$ is isomorphic to the doubly braced capped hexahedron in Figure \ref{fig_hex_disjoint}.
If $x\in R_1$ then, using the observation in the 
paragraph above, we see that there are no other vertices in the interior of the region $R_1\backslash T$.
If $x\not=v$ then the triangulation $P$ contains the edges $xv$, $xy_1$ and $xy_2$. 
Now note that the edge $xv$ is contractible in $G$, a contradiction. Thus, we conclude that $x=v$ and so \( V(G) = \{u,v,w\} \cup X_{uv} \). In other words \( G \) has six vertices and adjacent braces and so it is isomorphic to the example shown in Figure \ref{fig_hex_adjacent}.

Finally suppose that \( x \) also lies in the closed triangular region \( T \). 
We can construct a unibraced triangulation \( H \) by deleting the vertices 
\( v,y_3 \) and all their incident edges. Now \( H \) is a unibraced triangulation 
with a triangular face bounded by edges \( uy_1,y_1y_2,y_2u \). 
If \( H \) has six or more vertices, then by Lemma \ref{lem_cont_away_from_triangle}
there is some edge of \( H \) that is contractible that is not one 
of \( uy_1,y_1y_2,y_2u \). Such an edge would also be contractible in \( G \), 
contradicting our assumption that \( G \) is irreducible. Therefore 
\( H \) has only five vertices and it follows easily that \( G \) is isomorphic
to one of the examples shown in Figures \ref{fig_ireed_seven} or 
\ref{fig_ireed_seven_non}. This completes the proof of Theorem \ref{thm_irred_seven}.
\end{proof}

If \( G' = (P,B) \) is a braced triangulation and 
 \( e \) is an edge of \( P \) that is contractible in \( G' \), then
 \( G = (P/e,B) \) 
is a braced triangulation and we say that \( G' \) is obtained 
from \( G \) by a {\em topological vertex splitting move}.
 
Combining the above results we obtain the following theorem: 

\begin{thm}
    \label{thm_splitting_version}
    Let \( G \) be a doubly braced triangulation. Then \( G \) 
    can be constructed from one of the examples in
    Figures \ref{fig_doubly_braced_oct}, \ref{fig_hex_disjoint},
    \ref{fig_hex_adjacent}, \ref{fig_ireed_seven} or 
    \ref{fig_ireed_seven_non} by a sequence of 
    topological vertex splitting moves.
\end{thm}

In general, a {\em $d$-dimensional vertex splitting move} on a graph is defined as follows (see \cite{WWsplitting}, for example).
Let $G = (V, E)$ be a graph, and let $v_1\in V$ and $v_1v_i\in E$ for $i=2,\ldots, d$.
If a graph $G'$ is obtained from $G$ by 
 \begin{itemize}
\item adding a new vertex $v_0$  and edges $v_0v_1, v_0v_2, \ldots,  v_0v_{d}$ to $G$, and
\item  for every edge $v_1x\in E$ with $x\notin \{v_2, \ldots,  v_{d}\}$, either leaving the edge unchanged or replacing it with the edge $v_0x$,
\end{itemize}
then $G'$ is said to be obtained from $G$ by a \emph{$d$-dimensional vertex split} at $v_1$
 (on the edges $v_1v_2, \ldots , v_1v_{d}$). 

Of course, topological vertex splitting is a special case of ($3$-dimensional) vertex splitting for graphs, and it is known that vertex splitting for graphs preserves rigidity for generic frameworks in a wide variety of settings \cite{WWsplitting,dkn,CJ,jkn}. Therefore it is natural to look for geometric rigidity applications of Theorem \ref{thm_splitting_version}.
We provide two such applications in Sections \ref{sec:hypercylinder} and \ref{sec:mixednorms}.

\section{Application: rigidity in the hypercylinder}
\label{sec:hypercylinder}

There is a sizable literature on the rigidity of bar-joint frameworks in $3$-dimensional Euclidean space whose points
are constrained to lie on a surface (see, for example, \cite{jkn,jn,nop,nop1}.
 In this section, we will consider the
rigidity of bar-joint frameworks in $4$-dimensional Euclidean space where the 
points are constrained to lie on a hypercylinder.

\subsection{The hypercylinder in \( \mathbb R^4 \)}

Let \( \Sigma = \{(x,y,z,w) \in \mathbb R^4: x^2+y^2+z^2 = 1\}\) be
the hypercylinder in \( \mathbb R^4 \). Observe that 
\( \Sigma \) is a smooth three-dimensional manifold that 
inherits a natural metric  as a subspace of the Euclidean 
space \( \mathbb R^4 \). 
The group of isometries of \( \Sigma \) with respect to this 
metric is a Lie group of real dimension 4. 
Indeed this group is canonically 
isomorphic to \( O(3) \times E(1)  \) where \( O(3) \) is the 
group of \( 3\times 3 \) orthogonal matrices and \( E(1) \) is the 
group of Euclidean isometries of \( \mathbb R \).

Let $\mathcal{T}(\mathbb{R}^4)$ denote the real linear space of infinitesimal rigid motions of the Euclidean space $\mathbb{R}^4$. Recall that each infinitesimal rigid motion $\eta\in\mathcal{T}(\mathbb{R}^4)$ is an affine map $\eta:\mathbb{R}^4\to \mathbb{R}^4$ of the form $\eta(x) = B(x)+c$ where the linear part $B$ is a $4\times4$ skew-symmetric matrix and the translational part $c$ is a vector in $\mathbb{R}^4$.

Let \( \pi:\mathbb R^4 \rightarrow \mathbb R^3 \) be the projection 
\( (x,y,z,w) \mapsto (x,y,z) \) 
and let $\mathcal{T}(\Sigma)$ denote the following subspace of $\mathcal{T}(\mathbb{R}^4)$,
\[\mathcal{T}(\Sigma) = \{\eta\in\mathcal{T}(\mathbb{R}^4): \pi(\eta(x)) \cdot\pi(x)=0,\quad\forall\,x\in \Sigma\}.\]
We refer to the elements of $\mathcal{T}(\Sigma)$ as infinitesimal rigid motions of $\Sigma$.

\begin{lem}
\label{l:rigidmotions}
Let $\eta\in\mathcal{T}(\mathbb{R}^4)$ take the form $\eta(x) = B(x)+c$ where $B$ is a $4\times4$ skew-symmetric matrix and $c\in \mathbb{R}^4$.
Then $\eta\in \mathcal{T}(\Sigma)$ if and only if 
 $$B=\begin{bmatrix} \tilde{B} & 0 \\ 0 & 0 \end{bmatrix},$$
 where $\tilde{B}$ is a $3\times 3$ skew-symmetric matrix, and $\pi(c)=0$.
\end{lem}

\proof
Let $e_1,e_2,e_3,e_4$ denote the standard basis vectors in $\mathbb{R}^4$. 

If $\eta\in \mathcal{T}(\Sigma)$ then
$$\pi(c) = \sum_{i=1}^3 \left(\pi(\eta(e_i))\cdot \pi(e_i)\right)\pi(e_i)=0.$$
Also, note that $e_i+e_4\in\Sigma$ for $i=1,2,3$ and so,
$$B(e_4)\cdot e_i
=\pi(\eta(e_i+e_4))\cdot \pi(e_i+e_4) 
= 0.$$
Thus $B$ has the form,
 $$B=\begin{bmatrix} \tilde{B} & 0 \\ 0 & 0 \end{bmatrix},$$
 where $\tilde{B}$ is a $3\times 3$ skew-symmetric matrix.

For the converse, suppose the linear part of $\eta$ has the form
 $$B=\begin{bmatrix} \tilde{B} & 0 \\ 0 & 0 \end{bmatrix},$$
 where $\tilde{B}$ is a $3\times 3$ skew-symmetric matrix, and the translational part $c$ satisfies $\pi(c)=0$.
Since $\tilde{B}$ is skew-symmetric, if $x\in \Sigma$ then
$$\pi(\eta(x))\cdot\pi(x)
=\tilde{B}(\pi(x))\cdot\pi(x)=0.$$
Thus, $\eta\in \mathcal{T}(\Sigma)$.
\endproof

\subsection{Frameworks in the hypercylinder} For a graph \( G  = (V,E) \), a \emph{placement} of \( G  \) in 
\( \Sigma \) is a vector \( q=(q_v)_{v\in V}\in \Sigma^V \). A pair $(G,q)$ consisting of a graph $G$ and 
a placement $q$ is called
a \emph{(bar-joint) framework} in $\Sigma$. 
A {\em subframework} of $(G,q)$ is a framework $(H,q^H)$ where $H$ is a  subgraph of $G$ and $q^H_v=q_v$ for all $v\in V(H)$.

We say that $(G,q)$ is {\em full} in $\Sigma$ if the restriction map, $$\rho:\mathcal{T}(\Sigma)\to (\mathbb{R}^4)^V,\quad
\eta\mapsto (\eta(q_v))_{v\in V},$$ is injective.
In this case we refer to $q$ as a {\em full placement} of $G$ in $\Sigma$.
We say that $(G,q)$ is {\em completely full} in $\Sigma$ if every subframework of $(G,q)$ with at least $6$ vertices is full in $\Sigma$.

\begin{lem}
Let $(G,q)$ be a framework in $\Sigma$ and let $S=\{q_v:v\in V\}$.
Then $(G,q)$ is full in $\Sigma$ if and only if 
$\pi(S)$ contains at least two linearly independent vectors.
\end{lem}

\proof
Suppose $\pi(S)$ does not contain two linearly independent vectors. Let $A(\theta)\in O(3)$ be the rotation matrix with rotation axis spanned by $\pi(S)$, where $\theta$ denotes the angle of rotation. 
Let $B$ denote the skew-symmetric matrix  $$B=\begin{bmatrix} A'(0) & 0 \\ 0 & 0 \end{bmatrix}$$
Let $\eta\in\mathcal{T}(\mathbb{R}^4)$ be the infinitesimal rigid motion with $\eta(x)=B(x)$. Then, by Lemma \ref{l:rigidmotions},  $\eta\in \mathcal{T}(\Sigma)$. Also, the rotation axis for $A(\theta)$ lies in the kernel of $A'(0)$ and so $\eta(S)=(A'(0)(\pi(S)),0)=0$. 
Thus $(G,q)$ is not full.

For the converse, suppose there exists a non-zero $\eta\in \mathcal{T}(\Sigma)$ such that $\eta(S)=0$.
Note that, by Lemma \ref{l:rigidmotions},  $\eta(x)=\tilde{B}(\pi(x))+c$ for some non-zero $3\times 3$ skew-symmetric matrix $\tilde{B}$ and some $c=(0,0,0,w)\in \mathbb{R}^4$.
Now $\tilde{B}(\pi(S))=\pi(\eta(S))=0$. The rank of a skew-symmetric matrix is always even
and so the kernel of $\tilde{B}$ must have dimension $1$.
We conclude that $\pi(S)$ does not contain two linearly independent vectors.
\endproof

We denote by $\Full(G;\Sigma)$ the set of all full placements of $G$ in $\Sigma$. Note that by the above lemma, $\Full(G;\Sigma)$ is an open and dense subset of $\Sigma^V$.

\subsection{Rigidity in the hypercylinder}
Let $(G,q)$ be a framework in $\Sigma$.
An \emph{infinitesimal flex} of \( (G,q ) \) is a vector
\( m=(m_v)_{v\in V} \in (\mathbb R^4)^V \) that satisfies,
\begin{equation}
    (m_u-m_v)\cdot(q_u-q_v) = 0 \qquad \textrm{ for every } uv \in E, 
    \label{eq_edge}
\end{equation}
and,
\begin{equation}
    \pi(m_v)\cdot\pi(q_v) = 0 \qquad \textrm{ for every } v \in V.
    \label{eq_tang}
\end{equation} 
The constraints in (\ref{eq_edge}) are the standard Euclidean first-order length constraints 
for the edges of $G$, and the constraints in (\ref{eq_tang}) ensure that the velocity vectors of an infinitesimal flex 
lie in the tangent hyperplanes of $\Sigma$ at the corresponding points.

\begin{lem}
\label{lem:hcyltriv}
Let $(G,q)$ be a framework in $\Sigma$ and let $\eta\in\mathcal{T}(\Sigma)$.
Then the vector $m\in (\mathbb{R}^4)^V$ where,
$$m_v = \eta(q_v), \quad \forall \,v\in V,$$
is an infinitesimal flex of $(G,q)$.
\end{lem}

\proof
The conditions (\ref{eq_edge}) and (\ref{eq_tang}) are readily verified using Lemma \ref{l:rigidmotions}.
\endproof

We refer to the infinitesimal flexes described in Lemma $\ref{lem:hcyltriv}$ as {\em trivial} infinitesimal flexes of $(G,q)$. The set of all trivial infinitesimal flexes of $(G,q)$ is a linear subspace of $(\mathbb{R}^4)^V$, which we denote by $\mathcal{T}(q)$. 
The orthogonal projection of $(\mathbb{R}^{4})^V$ onto $\mathcal{T}(q)$ will be denoted $P_q$.
We say that \( (G ,q ) \) is \emph{infinitesimally rigid} if 
every infinitesimal flex of \( (G ,q ) \) is trivial.
 Otherwise $(G,q)$ is
\emph{infinitesimally flexible}.

\begin{lem}
\label{l:proj}
Let $G=(V,E)$ be a graph and let $x\in (\mathbb{R}^{4})^V$.
Then the map $$\phi_x:\Full(G;\Sigma)\to \mathcal{T}(q),\quad q\mapsto P_q(x),$$ is continuous.
\end{lem}

\proof
    If $(G,q)$ is full in $\Sigma$ then a basis for $\mathcal{T}(q)$ is given by
    the vectors $m_1(q),\ldots,m_4(q)$ where for each $v\in V$, we have
    $m_1(q)_v = (q_v^3,0,-q_v^1,0),\,
    m_2(q)_v = (q_v^2,-q_v^1,0,0),\,
    m_3(q)_v = (0,q_v^3,-q_v^2,0), \,
    m_4(q)_v = (0,0,0,1)$.
    The result follows since  $P_q(x)$ depends continuously on the basis vectors $m_1(q),\ldots,m_4(q)$.
    \endproof

\subsection{The rigidity matrix}
Let $(G,q)$ be a framework in $\Sigma$.
The \emph{rigidity matrix} \( R(G ,q ) \) for $(G,q)$ in $\Sigma$ is the matrix
corresponding to the linear system in  (\ref{eq_edge}) and (\ref{eq_tang}). This matrix is
a  \( (|E| + |V|) \times
4|V|\)  matrix of the following form.  The rows are indexed by the 
set \( E \cup V \) and the columns are indexed in collections of 
four by the set \( V \).  For an edge \( uv \in E \) the
corresponding row has entries \( q_u-q_v  \) in the collection
of columns corresponding to \( u \) and \( q_v -q_u \)
in the collection of columns corresponding to \( v \) and zeroes in all 
other columns. For a vertex 
\( v \in V\) the corresponding row has entries 
\( (\pi(q_v),0) \) in the collection of columns indexed by \( v \) and
zeroes in all other columns.

\begin{lem}
\label{lem:hcylrank}
Let \( (G,q)  \) be a full framework in $\Sigma$.
Then \( (G ,q ) \) is infinitesimally rigid if and only if 
\( \rank R(G ,q ) = 4|V| - 4 \).
\end{lem}

\proof
Note that the kernel of $R(G,q)$ is the linear space of infinitesimal flexes of $(G,q)$. Thus,  \( (G ,q ) \) is infinitesimally rigid if and only if 
$\ker R(G,q)=\mathcal{T}(q)$. Also, note that $\rank R(G,q)=4|V|-\dim \ker R(G,q)$.
Since $(G,q)$ is full in $\Sigma$, $\dim \mathcal{T}(q)=\dim \mathcal{T}(\Sigma)=4$. The result now follows.
\endproof

All of the above discussion is by way
of context for the following result, which provides necessary conditions
for a full framework in $\Sigma$ to be {\em minimally} infinitesimally rigid. We say that 
a graph $G=(V,E)$ is \emph{ \( (3,4) \)-tight} if $|E|=3|V|-4$ and  $|E'|\leq 3|V'|-4$ for every subgraph 
$G'=(V',E')$ containing at least one edge.

\begin{thm}
    \label{thm_hypercyclinder_necessary}
    Suppose that \( G=(V,E)  \) has at least six vertices
    and that  \( (G ,q ) \) is completely full and  infinitesimally
    rigid in $\Sigma$. Furthermore suppose that for any \( e \in E \),
    \( (G -e,q ) \) is not infinitesimally rigid. Then 
    \( G  \) is \( (3,4) \)-tight.
\end{thm}
\begin{proof} 
Since $(G,q)$ is full in $\Sigma$ we have $\dim \mathcal{T}(q)=4$.
If $|E|<3|V|-4$, then the rigidity matrix
$R(G,q)$ has rank less than $4|V|-4$. It follows that $\mathcal{T}(q)$ is a proper subspace of $\ker R(G,q)$ and so $(G,q)$ is
infinitesimally flexible, a contradiction. If $|E|>3|V|-4$, then 
$R(G,q)$ has a non-trivial row dependence $\omega\in \mathbb{R}^{E\,\cup\, V}$.
By the structure of $R(G,q)$, the rows of $R(G,q)$ indexed by $V$ 
are linearly independent, and hence $\omega_e\neq 0$ for some edge $e\in E$.
It follows that the removal of the edge $e$ does not decrease the rank of the
rigidity matrix and so $(G-e,q)$ is still infinitesimally rigid, a contradiction.

Similarly, if there is a non-trivial subgraph $G'=(V',E')$ with $|E'|> 3|V'|-4$, then,
by the simplicity of $G$,
$|V'|\geq 6$. 
Since $(G,q)$ is completely full in $\Sigma$, the subframework $(G',q_{G'})$ is full in $\Sigma$.
The $(|E'|+|V'|)\times 4|V'|$ submatrix of $R(G,q)$
corresponding to $G'$ has a non-trivial row dependence with a non-zero support on one of the edges of $G'$. Thus, as above, it follows that the removal 
of this edge from $G$ leaves the framework infinitesimally rigid, a contradiction.
 This gives the result.
\end{proof}

On the other hand we may ask if, given a \( (3,4) \)-tight
graph \( G  \), there is a placement \( q  \) of 
\( G  \) in \( \Sigma \) such that \( (G ,q ) \) is minimally infinitesimally 
rigid. In general this is open. However, in the following section we show this to be true whenever \( G  \) is the underlying graph of a doubly braced
triangulation. 

\subsection{Minimal rigidity of doubly braced triangulations}
We say that a placement \( q\in\Sigma^V\) of a graph \( G=(V,E)  \)
in \( \Sigma \) is \emph{regular} if 
the function,
\[r_G:\Sigma^V\to\mathbb{N},\quad x\mapsto \rank R(G,x)\]
achieves its maximum value at $q$.
Note that the set of regular placements of $G$ in $\Sigma$ is an open and dense subset of $\Sigma^V$. Moreover, if $(G,q)$ is infinitesimally rigid (respectively, flexible) in $\Sigma$ for some regular placement $q$ then every regular placement of $G$ in $\Sigma$ is infinitesimally rigid (respectively, flexible).
In this case, we say that the graph $G$ is rigid (respectively, flexible) in $\Sigma$.

\begin{lem}
    \label{lem_seven_rigid}
    The graph \( K_5\cup_{K_3}K_5 \)  is (minimally) rigid in the hypercylinder $\Sigma$.
\end{lem}

\begin{proof}
    By Lemma \ref{lem:hcylrank}, it suffices to find a particular placement of the graph 
    whose associated rigidity matrix has rank 24. Since a randomly chosen 
    matrix will, with probability 1, yield a rigidity matrix with maximum rank
    it is easy to find such a placement. For example we have verified that the 
    following placement yields the required rank:
    $$q_1 = \frac1{\sqrt3}(1,1,1,1) \quad\quad
    q_2 = \frac1{\sqrt{17}}(3,2,2,1) \quad\quad
    q_3 = \frac1{\sqrt{17}}(2,3,2,3)$$ $$
    q_4 = \frac1{\sqrt{11}}(1,3,1,1)
    \quad\quad\quad
    q_5 = \frac1{\sqrt{17}}(3,2,2,2)$$
    $$ q_6 = \frac1{\sqrt{14}}(2,3,1,1) 
    \quad\quad\quad
    q_7 = \frac1{\sqrt{14}}(2,1,3,2) $$
        where one $K_5$ is induced by $q_1,\dots,q_5$ and the other is induced by 
    $q_3,\dots,q_7$.
\end{proof}

\begin{lem}
    \label{lem_K6minusedge}
   The graph  \( K_6-e \)  is (minimally) rigid in the hypercylinder $\Sigma$.
\end{lem}

\begin{proof}
As for Lemma \ref{lem_seven_rigid}, it suffices to find one placement 
that yields a rigidity matrix of rank 20. In this case the placement 
$$q_1 = \frac1{\sqrt{11}}(1,3,1,1) 
\quad\quad
 q_2 = \frac1{3}(2,2,1,1) 
 \quad\quad
 q_3 = \frac1{\sqrt{22}}(3,2,3,3) $$
$$ q_4 = \frac1{\sqrt{22}}(3,3,2,2) 
\quad\quad
 q_5 = \frac1{\sqrt{27}}(3,3,3,1) 
 \quad\quad
 q_6 = \frac1{\sqrt{14}}(1,2,3,2) $$
where the missing edge is between $q_5$ and $q_6$, 
yields the required rank.
\end{proof}

For each $k\in\mathbb{N}$ define,
$$H_k = \left\{x=(x_1,x_2,x_3,x_4)\in\mathbb{R}^4:\, x_1^2+x_2^2+x_3^2<\frac{1}{k^4}\, \mbox{ and }\,\frac{1}{2k}<x_4<\frac{1}{k}\right\}.$$
Note that $H_k$ is the interior of a truncated hypercylinder with radius $\frac{1}{k^2}$ and height $\frac{1}{2k}$. 

\begin{lem}
\label{lem_Hk}
Let $k\in\mathbb{N}$, let $x\in H_k$ and let $e_4=(0,0,0,1)\in\mathbb{R}^4$.
Then 
$$\left\Vert\frac{x}{\|x\|} -e_4\right\Vert<\frac{2\sqrt{2}}{k}.$$
\end{lem}

\proof
Note that $0<x_4\leq \|x\|$ and $\frac{1}{\|x\|}<2k$. We have,
\begin{eqnarray*}
\left\Vert\frac{x}{\|x\|} -e_4\right\Vert^2
&=& \left\Vert\frac{(x_1,x_2,x_3,x_4-\|x\|)}{\|x\|}\right\Vert^2 \\
&<& \left(\frac{1}{k^4}+(x_4-\|x\|)^2\right)4k^2 \\
&=& \left(\frac{1}{k^4}+\|x\|^2+x_4^2-2x_4\|x\|\right)4k^2 \\
&\leq& \left(\frac{1}{k^4}+\|x\|^2-x_4^2\right)4k^2 \\
&=& \left(\frac{1}{k^4}+x_1^2+x_2^2+x_3^2\right)4k^2 \\
&<& \frac{8}{k^2}.
\end{eqnarray*}

\endproof

\begin{prop}
    \label{prop_split_rigid}
    Suppose that \( G  \) is (minimally) rigid in the hypercylinder $\Sigma$ and that 
    \( G ' \) is obtained from \( G  \) by a $3$-dimensional vertex splitting 
    move. Then  \( G '\)  is also (minimally) rigid in $\Sigma$ 
\end{prop}

\begin{proof}
We adapt the proof of Lemma 5.1 in \cite{nop}.

Suppose $G=(V,E)$ has $n$ vertices $v_1, v_2,\ldots, v_n$. Let $G'=(V',E')$ be obtained from $G$ by a 3-dimensional vertex splitting move at the vertex $v_1$ on the edges $v_1v_2$ and $v_1v_3$. Let $V'=V\cup \{v_0\}$.
We will show that if $G'$ is flexible in $\Sigma$ then $G$ is also  flexible in $\Sigma$.

Suppose $G'$ is  flexible in $\Sigma$ and
let $q\in \Sigma^V$ be a regular placement of $G$ in $\Sigma$. For convenience we will write, 
$$q=(q_{v_1},q_{v_2},\ldots, q_{v_n})=(q_1,q_2,\ldots, q_n).$$
Let $n_1\in \mathbb{R}^4$ be a normal vector to the tangent plane of $\Sigma$ at $q_1$
and let $e_4=(0,0,0,1)\in \mathbb{R}^4$. 
Since the set of regular placements of $G$ in $\Sigma$ is open in $\Sigma^V$ we may assume that the vectors $q_1-q_2$,  $q_1-q_3$, $n_1$ and $e_4$ are linearly independent in $\mathbb{R}^4$.

Define $q'_v=q_v$ for all $v\in V$ and $q'_{v_0} =q_{v_1}$. Then $q'=(q'_v)_{v\in V'}$ is a non-regular placement of $G'$ in $\Sigma$. Again for convenience we write,
$$q'=(q'_{v_0},q'_{v_1},q'_{v_2},\ldots, q'_{v_n})=(q_1,q_1,q_2,\ldots, q_n).$$
For each $k\in\mathbb{N}$, let $B_k$ denote the open ball in $\mathbb{R}^4$ with centre $0$ and radius $\frac{1}{k}$ and consider the following subset of $\mathbb{R}^{4(n+1)}$,
$$U_k = H_k\times \overbrace{B_k\times \cdots \times B_k}^n.$$
Let $N_k = (q'+U_k)\cap \Sigma^{V'}$ and note that $N_k$ is a non-empty open subset of $\Sigma^{V'}$. 
Since the set of regular placements of $G'$ in $\Sigma$ is dense in $\Sigma^{V'}$, for each $k\in\mathbb{N}$ there exists a regular placement  $q^k$ of $G'$ in $\Sigma$ such that $q^k\in N_k$. 
Moreover, by applying an isometry of $\Sigma$ to the components of $q^k$ we may assume that $q^k_{v_1}=q_{v_1}$ for each $k\in\mathbb{N}$. 
For convenience we  write, 
$$q^k=(q^k_{v_0},q^k_{v_1},\ldots, q^k_{v_n})=(q^k_0,q_1,q^k_2,\ldots, q^k_n).$$

Note that the sequence $(q^k)$ of regular placements of $G'$ in $\Sigma$ converges to the non-regular placement $q'$.
Also note that for each $k\in\mathbb{N}$, we have $q^k-q'\in U_k$.
In particular, $q^k_0-q_1\in H_k$ and so, by Lemma \ref{lem_Hk}, 
$$\left\Vert\frac{q_0^k-q_1}{\|q_0^k-q_1\|} -e_4\right\Vert < \frac{2\sqrt{2}}{k}.$$
It follows that the sequence of unit vectors $\frac{q_0^k-q_1}{\|q_0^k-q_1\|}$ converges to $e_4=(0,0,0,1)\in\mathbb{R}^4$.

For each $k\in\mathbb{N}$, the framework $(G',q^k)$ is infinitesimally flexible in $\Sigma$ and so there exists a unit vector $m^k=(m_0^k,m_1^k,\ldots, m_n^k)\in (\mathbb{R}^4)^{V'}$ which is a non-trivial infinitesimal flex of $(G',q^k)$.
We may assume, without loss of generality, that $m^k$ has no trivial flex component, in the sense that $P_{q^k}(m^k)=0$. 
By passing to a subsequence (using the Bolzano-Weierstrass Theorem), we may assume that the sequence $(m^k)$ converges to a unit norm vector $m'=(m_0,m_1,\ldots, m_n)\in (\mathbb{R}^4)^{V'}$. 
Note that for each edge $v_iv_j$ in $G'$, we have,
$$m_i\cdot (q'_{v_i}-q'_{v_j})= \lim_{k\to \infty} m_i^k\cdot(q_{v_i}^k-q_{v_j}^k)= 
     \lim_{k\to \infty} m_j^k\cdot(q_{v_i}^k-q_{v_j}^k)=m_j\cdot(q'_{v_i}-q'_{v_j}),$$
and for each vertex $v_i$ in $G'$ we have,
$$\pi(m_i)\cdot\pi(q'_{v_i}) = \lim_{k\to \infty} \pi(m_i^k)\cdot\pi(q_{v_i}^k)=0.$$
Moreover, by Lemma \ref{l:proj},
       $$P_{q'}(m') = \lim_{k\to\infty} P_{q^k}(m^k)=0.$$  Thus, $m'$ is a non-trivial infinitesimal flex of $(G',q')$.

We claim that $m_0=m_1$. To see this, note that since $m'$ is an  infinitesimal flex of $(G',q')$ we have, for $i=2,3$,
$$m_1\cdot (q_1-q_i)=m_1\cdot (q'_{v_1}-q'_{v_i})=m_i\cdot (q'_{v_1}-q'_{v_i})=m_i\cdot(q_1-q_i),$$
     $$m_0\cdot (q_1-q_i)=m_0\cdot (q'_{v_0}-q'_{v_i})=m_i\cdot (q'_{v_0}-q'_{v_i})=m_i\cdot(q_1-q_i).$$
     Thus,  $(m_0-m_1)\cdot (q_1-q_i)=0$ for $i=2,3$.
     We also have,
      $$(m_0-m_1)\cdot n_1 = \pi(m_0)\cdot \pi(q'_{v_0})-\pi(m_1)\cdot \pi(q'_{v_1})=0,$$
      and since $m^k$ is an infinitesimal flex of $(G',q^k)$, 
      $$(m_0-m_1)\cdot e_4 = 
      \lim_{k\to \infty}\, (m_0^k-m_1^k)\cdot \frac{q_0^k-q_1}{\|q_0^k-q_1\|}=0.$$
      Thus, $m_0-m_1$ is orthogonal to the four linearly independent vectors $q_1-q_2$,  $q_1-q_3$, $n_1$ and $e_4$ and hence $m_0=m_1$.
      
      It now follows that the vector $m=(m_1,m_2,\ldots, m_n)$ is a non-trivial infinitesimal flex of $(G,q)$. 
      We conclude that $G$ is  flexible in $\Sigma$.
\end{proof}

\begin{thm}
\label{t:hypercylinder}
    Let \( G \) be the graph of a doubly braced triangulation. 
		Then \( G  \) is (minimally) rigid in the hypercylinder $\Sigma$.
\end{thm}

\begin{proof}
    This follows immediately from Theorem \ref{thm_irred_seven},
    Lemmas \ref{lem_seven_rigid}, \ref{lem_K6minusedge} and Proposition 
    \ref{prop_split_rigid}.
\end{proof}


\section{Application: rigidity for mixed norms on $\mathbb{R}^3$}
\label{sec:mixednorms}

The rigidity theory of bar-joint frameworks in non-Euclidean finite dimensional real normed linear spaces was first considered in \cite{kit-pow}. This and subsequent work has explored special classes of norms, particularly the classical $\ell_p$-norms, polyhedral norms, unitarily invariant matrix norms and product norms (eg.~\cite{dkn,kitson,kitlev,kit-pow,kit-sch}). In this section, we consider a new context provided by a class of {\em mixed norms} on $\mathbb{R}^3$.

\subsection{The normed space $\ell_{2,p}^3$}
For $p\in(1,\infty)$, define the {\em mixed $(2,p)$-norm} 
on \( \mathbb R^3 \) by, 
\[ \|(x,y,z)\|_{2,p} = ((x^2+y^2)^{\frac{p}{2}}+|z|^p)^{\frac{1}{p}}.\]
We denote the normed spaces $(\mathbb{R}^3,\|\cdot\|_{2,p})$ by $\ell_{2,p}^3$. 
Note that the $(2,2)$-norm is the standard Euclidean norm on $\mathbb{R}^3$. Our main interest will be the non-Euclidean $(2,p)$-norms (i.e.~when $p\not=2$). 
The main result in this section states that the graph of a doubly braced triangulation is minimally rigid in $\ell_{2,p}^3$ for all $p\in(1,\infty)$, $p\not=2$ (Theorem \ref{t:norms}).

\begin{rem}
We have excluded the extreme case where $p=1$ as our geometric techniques are not applicable in that setting. In particular, the $(2,1)$-norm is neither smooth nor strictly convex (note that the unit sphere is a double cone). For similar reasons we will not consider the $(2,\infty)$-norm,
\[ \|(x,y,z)\|_{2,\infty} = \max \{(x^2+y^2)^{\frac{1}{2}},|z|\}.\]
(Note that in this case the unit sphere is cylindrical.)
Rigidity theory for the $(2,\infty)$-norm is developed in \cite[\S5.1]{kitlev} using different techniques. 
\end{rem}

Our first goal is to collect some preliminary geometric results which will be required later on.

\begin{lem}
\label{l:dual}
Let $p\in(1,\infty)$. 
Then the dual space of $\ell_{2,p}^3$ is $\ell_{2,q}^3$, where $q$ satisfies $\frac{1}{p}+\frac{1}{q}=1$. 
\end{lem}

\proof
Given $y\in\ell_{2,q}^3$, define $f_y: \ell_{2,p}^3\rightarrow \mathbb{R}$, $x\mapsto x\cdot y$. Note that for every $x\in\ell_{2,p}^3$,  the Cauchy-Schwarz and H{\"o}lder inequalities imply that
\[ |f_y(x)|\leq \sum_{i=1}^3 |x_iy_i|\leq \left\|\begin{bmatrix}x_1 \\x_2 \end{bmatrix}\right\|_2\,\left\|\begin{bmatrix}y_1 \\y_2 \end{bmatrix}\right\|_2 + |x_3y_3|\leq \left\| \begin{bmatrix}x_1 \\x_2 \\x_3\end{bmatrix}\right\|_{2,p} \left\| \begin{bmatrix}y_1 \\y_2 \\y_3\end{bmatrix}\right\|_{2,q}. \]
 Hence it suffices to show that the contraction, 
\[ T: \ell_{2,q}^3 \rightarrow (\ell_{2,p}^3)^\ast,\quad y\mapsto f_y,\]
is an isometry. 
Let $y\in \ell_{2,q}^3$ be non-zero. There exists $\theta\in [0,2\pi)$ such that, 
\[  y^\theta= R_\theta y= \begin{bmatrix}y^\theta_1 \\0 \\ y^\theta_3\end{bmatrix},\]
where $R_\theta$ is the isometry given by clockwise rotation by $\theta$ about the $z$-axis. Choose $x_\theta$, such that $x_k^\theta=|y_k^\theta|^{q-1} \sgn(y_k^\theta)$. 
Note that $f_y(R_{-\theta}x^\theta) = y\cdot R_{-\theta}x^{\theta} = R_\theta y\cdot x^\theta = f_{y^\theta}(x^\theta)$. Also, 
\[ \|x^\theta\|_{2,p}=\left(\sum_{k=1}^3 |y_k^\theta|^q\right)^{1/p}.\]
Hence we have,
\[\|f_y\|_{2,p}^*\geq \frac{f_y(R_{-\theta}x^\theta)}{\|R_{-\theta}x^\theta\|_{2,p}}= \frac{f_{y^\theta}(x^\theta)}{\|x^\theta\|_{2,p}}=\frac{\sum_{k=1}^3 |y_k^\theta|^q}{\left(\sum_{i=1}^3 |y_k^\theta|^q\right)^{1/p}}=
\left(\sum_{k=1}^3 |y_k^\theta|^q\right)^{1/q}
=\|y_\theta\|_{2,q}=\|y\|_{2,q},\]
and so $\|f_y\|_{2,p}^* = \|y\|_{2,q}$.
\endproof

\begin{lem}\label{lem:smooth}
The space $\ell_{2,p}^3$ is smooth and strictly convex for every $p\in (1,\infty)$.
\end{lem}

\proof 
By \cite[Lemma 1]{kit-sch}, it suffices to show that for all non-zero $(x,y,z)$ and $(a,b,c)$ in $\ell_{2,p}^3$, the function
\[\zeta:\mathbb{R}\to \mathbb{R},\quad t\mapsto \|(x,y,z)+t(a,b,c)\|_{2,p}\]
is differentiable at zero. Note that,
\[\zeta=h\circ\left( f+g\right),\]
where $f(t)=\left((x+ta)^2+(y+tb)^2\right)^{\frac{p}{2}}$, $g(t)=|z+tc|^p$ and $h(t)=t^\frac{1}{p}\,(t>0)$. 
Applying the chain rule, it suffices to show that $f,g$ are differentiable at $0$. We shall prove it only for $f$, since the same arguments will work for $g$. When $(x,y)\neq(0,0)$, then $(x+ta)^2+(y+tb)^2>0$ for $t$ sufficiently close to zero, and hence $f$ is differentiable. So it remains to check the case for $(x,y)=0$. Then 
\[ \left|\frac{\left((ta)^2+(tb)^2\right)^{p/2}}{t}\right|=  |t|^{p-1}(a^2+b^2)^{p/2}.\] 
Since \[\lim_{t\to 0} |t|^{p-1}(a^2+b^2)^{p/2}= 0,\]
it follows that $f$ is differentiable with $f'(0)=0$.

By Lemma \ref{l:dual}, for each $p\in (1,\infty)$ the space $\ell_{2,p}^3$ is reflexive and the dual of a smooth space. Thus $\ell_{2,p}^3$ is also strictly convex (see \cite[p.~184]{beau} eg.).
\endproof

\subsection{Isometries of $\ell_{2,p}^3$}
Next, we determine the identity component $\operatorname{Isom}_0(\ell_{2,p}^3)$ of the isometry group $\operatorname{Isom}(\ell_{2,p}^3)$ for  $p\neq 2$. It is known, by the Mazur-Ulam Theorem (\cite[Theorem 3.1.2]{thompson}), that  for every isometry $\phi$ on a real finite dimensional normed space $X$, there exists a linear isometry $T_\phi$ on $X$ and $t_\phi\in X$, such that  
\[\phi(x)=T_\phi x+t_\phi , \text{ for all }x\in X.\]
Moreover, the map $\phi \mapsto T_\phi$ is a group homomorphism with kernel equal to the group of translations on $X$.
 Hence we can focus on linear isometries. To do this, we recall John's theorem regarding the L{\"o}wner-John ellipsoid, that is the ellipsoid of maximal volume, inside the unit ball of $X$ (see \cite{joh} and \cite[Theorem 3.3.1]{thompson}). 

\begin{thm}[John]\label{thm:John}
Each convex body $K$ in $\mathbb{R}^n$ contains a unique ellipsoid of maximal volume, called the inner L{\"o}wner-John ellipsoid. This ellipsoid is equal to the Euclidean unit ball $B_2^n$ if and only if $B_2^n$ is contained in $K$ and there exist unit vectors $u_i \in \partial K$ and positive numbers $c_i$, $i=1,\dots,m$, such that:

\begin{enumerate}[(i)]
	\item $\sum\limits_{i=1}^{m} c_i u_i=0$;
	\item  $\sum\limits_{i=1}^{m} c_i \langle  x,u_i\rangle^2=\|x\|_2^2$ for all $x\in \mathbb{R}^n$.
\end{enumerate}
\end{thm}

\begin{cor}\label{isomxsubeu}
Let $X$ be a normed space that is associated with an inner L{\"o}wner-John ellipsoid $E$. Then the group $\operatorname{Isom}(X)$ of isometries of $X$ is a subgroup of the isometry group of the Euclidean space with unit ball $E$.
\end{cor}

\proof
Given a linear isometry $T$ on $X$, note that $T$ maps $B_X$ to itself. Since $T$ is volume preserving, it follows by the uniqueness of the inner L{\"o}wner-John ellipsoid that $T(E)=E$, so $T$ is also an isometry of the Euclidean space associated with $E$.
\endproof
 
Let $B_{2,p}^3\,$ denote the closed unit ball in $\ell_{2,p}^3$ and let $B_2^3$ denote the closed unit ball in Euclidean space $\mathbb{R}^3$.

\begin{lem}\label{lem:LJell}
Let $p\in(2,\infty)$. Then the inner L{\"o}wner-John ellipsoid for $B_{2,p}^3$ is  $B_2^3$.
\end{lem}
\proof
We apply Theorem \ref{thm:John} with $K=B_{2,p}^3$.
Let $\{e_1,e_2,e_3\}$ be the standard orthonormal basis in $\mathbb{R}^3$. Note that for $p>2$, $B_2^3$ is contained in  $B_{2,p}^3$ and each vector $e_i$ lies on $\partial B_{2,p}^3$, $i=1,2,3$.  
Define for $i=1,2,3$ the vectors $u_i=e_i$, $u_{i+3}=-e_i$ and the scalars $c_i=\frac{1}{2}$, $c_{i+3}=c_i$.  Property $(i)$ of Theorem \ref{thm:John} is evident, while property $(ii)$ is satisfied by Parseval's identity. The result follows.
\endproof

Recall that the orientation preserving isometries on the Euclidean space $\mathbb{R}^n$ are of the form $\phi(x)= T_\phi x+t_\phi$ with $T_\phi\in \operatorname{SO}(n)$, meaning that $\operatorname{det}T_\phi=1$. Hence the identity component $\operatorname{Isom}_0(\mathbb{R}^n)$ is generated by translations and rotations.

\begin{prop}\label{prop:orprisom}
Let $p\in (1,\infty)$, $p\not=2$. Then $\operatorname{Isom}_0(\ell_{2,p}^3)$ is the group generated by rotations $R$ about the $z$-axis and translations $T_t$,  $t\in \ell_{2,p}^3$.   \end{prop}

\proof
Let $T$ be a linear isometry that lies in $\operatorname{Isom}_0(\ell_{2,p}^3)$. We consider first the case $p> 2$.  It follows by Corollary \ref{isomxsubeu} and Lemma \ref{lem:LJell} that the linear isometries of $\ell_{2,p}^3$ are a subgroup of the group of linear isometries of Euclidean space $\ell_2^3$. Hence $T$ leaves invariant the set 
\[\partial B_2^3\cap \partial B_{2,p}^3=\{(x,y,0):x^2+y^2=1\}\cup\{(0,0,\pm 1)\}.\]
Since $T$ fixes both poles $(0,0,\pm 1)$ it also fixes the $z$-axis.  Thus, $T$ is a rotation operator about the $z$-axis. 

Now suppose $p\in(1,2)$. Note that in this case the dual operator $T^\ast$ is a linear isometry on $\ell_{2,q}^3$, where $\frac{1}{p}+\frac{1}{q}=1$. Moreover, $T^\ast$ lies in the identity component $\operatorname{Isom}_0(\ell_{2,q}^3)$. Since $q=1+\frac{1}{p-1}>2$, the above argument shows that $T^\ast$ is a rotation operator about the $z$-axis. It follows that $T$ is also a rotation operator about the $z$-axis.
\endproof


\subsection{Rigid motions of $\ell_{2,p}^3$}
Let $X$ be a finite dimensional real normed linear space. A {\em rigid motion} of $X$ is a collection $\alpha=\{\alpha_x:[-1,1]\to X\}_{x\in X}$  with the properties that:
\begin{enumerate}[(i)]
\item $\alpha_x$ is a continuous path, for all $x\in X$;
\item $\alpha_x(0)=x$ for any $x\in X$;
\item $\|\alpha_x(t)-\alpha_y(t)\|=\|x-y\|$ for all $x,y\in X$ and all $t\in [-1,1]$. 
 \end{enumerate}

\begin{prop}\label{prop:rigid}
Let $X$ be a finite dimensional real normed linear space and let $\alpha=\{\alpha_x\}_{x\in X}$ be a rigid motion of $X$. For each $t\in [-1,1]$, define
$$\beta_t:X\to X,\quad \beta_t(x) = \alpha_x(t)-\alpha_0(t).$$
Then,
\begin{enumerate}[(i)]
\item $\beta_t\in \Isom_0(X)$ for each $t\in [-1,1]$.
\item The map $\beta: [-1,1]\rightarrow \operatorname{Isom}_0(X)$, $t\mapsto \beta_t$,
is continuous.
\end{enumerate}
\end{prop}

\proof
Note that, for each $t\in [-1,1]$, $\beta_t$ is isometric and $\beta_t(0)=0$. It follows, by the Mazur-Ulam Theorem, that $\beta_t$ is a linear isometry. 
Let $t_0\in[-1,1]$ and let $\epsilon>0$. Since the unit ball $B_X$ is compact, we can choose $x_1,x_2,\dots,x_n\in B_X$, such that 
\[B_X\subseteq\bigcup_{i=1}^n \,B\left(x_i,\frac{\epsilon}{4}\right)\quad \text{and}\quad 0\in\{x_i\}_{i=1}^n.\]
Since the paths $\alpha_{x_1},\ldots,\alpha_{x_n}$ are continuous we can choose $\delta>0$ such that for all $t\in[-1,1]$, 
\[|t-t_0|<\delta \quad \implies \quad \max_{1\leq i\leq n}\,\|\alpha_{x_i}(t)-\alpha_{x_i}(t_0)\|<\frac{\epsilon}{4}\]
Let $x\in B_X$. Then there exists $i_0\in\{1,2,\dots,n\}$ such that $\|x-x_{i_0}\|_X\leq\frac{\epsilon}{4}$.
For each $t\in[-1,1]$ we have,
\begin{eqnarray*}
\|\alpha_x(t)-\alpha_x(t_0)\|
&\leq& \|\alpha_x(t)-\alpha_{x_{i_0}}(t)\|+\|\alpha_{x_{i_0}}(t)-\alpha_{x_{i_0}}(t_0)\|+\|\alpha_{x_{i_0}}(t_0)-\alpha_x(t_0)\| \\
&\leq& 2\|x-x_{i_0}\|+\frac{\epsilon}{4}\leq \frac{3\epsilon}{4}
\end{eqnarray*}
Hence for all $t\in (t_0-\delta,t_0+\delta)$ we have
\begin{eqnarray*}
\|\beta_t(x)-\beta_{t_0}(x)\|
&=&\|(\alpha_x(t)-\alpha_0(t))-(\alpha_x(t_0)-\alpha_0(t_0))\| \\ 
&\leq& \|\alpha_x(t)-\alpha_x(t_0)\|+\|\alpha_0(t)-\alpha_0(t_0)\| \\ 
&\leq& \frac{3\epsilon}{4}+\frac{\epsilon}{4}=\epsilon.
\end{eqnarray*}
Since $\delta$ is independent of $x\in B_X$, it follows that for all $t\in[-1,1]$, 
\[|t-t_0|<\delta \quad \implies \quad \|\beta_t-\beta_{t_0}\|_{op}<\epsilon.\]
Thus the map $\beta: [-1,1]\rightarrow \operatorname{Isom}(X)$, $t\mapsto \beta_t$,
is continuous.

Finally, note that $\beta([-1,1])$ is a connected subset of $\Isom(X)$ which contains the identity on $X$. Hence $\beta_t$ lies in $\operatorname{Isom}_0(X)$ for all $t\in [-1,1]$.

\endproof

\begin{cor}\label{cor:rigid}
Let $p\in(1,\infty)$, $p\not=2$.
A collection of maps $\alpha=\{\alpha_x:[-1,1]\to \ell_{2,p}^3\}_{x\in \ell_{2,p}^3}$ is a rigid motion of $\ell_{2,p}^3$  if and only if there exists a continuous map $\theta:[-1,1]\to \mathbb{R}$ which 
satisfies $\theta(0)=0$ such that for each $x=(x_1,x_2,x_3)\in\ell_{2,p}^3$ and $t\in[-1,1]$,
\begin{equation}
\alpha_x(t) = \begin{bmatrix} \cos\theta(t) & -\sin\theta(t) & 0 \\  \sin\theta(t) & \cos\theta(t) & 0\\0 & 0 & 1 \end{bmatrix} \begin{bmatrix} x_1\\x_2\\x_3\end{bmatrix}+\alpha_0(t).
\label{rigidl2p3}
\end{equation}
\end{cor}

\proof
Suppose $\alpha$ is a rigid motion of $\ell_{2,p}^3$.
Then equation \eqref{rigidl2p3} follows directly from Propositions \ref{prop:orprisom} and \ref{prop:rigid}.
Moreover, since $\alpha_x(t)$ is continuous on $[-1,1]$
, the same also holds for the map $t\mapsto\theta(t)$. Note also that we can take $\theta(0)=0$.
The converse direction is clear.
\endproof

Let $\alpha=\{\alpha_x\}_{x\in X}$ be a rigid motion of a normed space $X$. 
If each $\alpha_x$ is differentiable at $t=0$ then the map $\eta:X\to X$, $\eta(x)=\alpha_x'(0)$, is called an {\em infinitesimal rigid motion} of $X$. 
The collection of all infinitesimal rigid motions of $X$ is a real vector space, denoted $\mathcal{T}(X)$.

\begin{thm}\label{th:infrmdim}
Let $p\in(1,\infty)$, $p\not=2$, and let $\eta:\ell_{2,p}^3\to \ell_{2,p}^3$ be an affine map.
Then $\eta\in\mathcal{T}(\ell_{2,p}^3)$
if and only if there exists a scalar $\lambda\in\bR$ and a vector $c\in \mathbb{R}^3$ such that,
\[\eta(x_1,x_2,x_3) = \lambda(-x_2,x_1,0)+c,\]
 for all $(x_1,x_2,x_3)\in \mathbb{R}^3$.

In particular, $\dim \mathcal{T}(\ell_{2,p}^3) = 4$.
\end{thm}

\proof
By Corollary \ref{cor:rigid}, if $\eta$ is an infinitesimal rigid motion of $\ell_{2,p}^3$ then there exists $\theta$ such that, for each $x\in \bR^3$, $\eta(x)$ is given by,  
\[ \eta(x) = \frac{d}{dt}\left(\begin{bmatrix} \cos\theta(t) & -\sin\theta(t) & 0 \\
 \sin\theta(t) & \cos\theta(t) & 0\\0 & 0 & 1 \end{bmatrix} \begin{bmatrix} x_1\\x_2\\x_3\end{bmatrix}+\alpha_0(t)\right)\bigg|_{t=0}=
 \theta'(0)\begin{bmatrix} -x_2\\ x_1\\0 \end{bmatrix}+\alpha_0'(0).\]

For the converse, suppose $\eta$ is an affine map of the form 
\[\eta(x_1,x_2,x_3) = \lambda(-x_2,x_1,0)+c,\]
for some scalar $\lambda\in\bR$ and some vector $c\in \mathbb{R}^3$.
Consider the collection of continuous paths,
\[\alpha_x(t) = \begin{bmatrix} \cos(\lambda t) & -\sin(\lambda t) & 0 \\
 \sin(\lambda t) & \cos(\lambda t) & 0\\0 & 0 & 1 \end{bmatrix} \begin{bmatrix} x_1\\x_2\\x_3\end{bmatrix}  + tc.\]
Then $\alpha$ is a rigid motion of $\ell_{2,p}^3$ that satisfies $\eta(x)=\alpha_x'(0)$ for each $x\in \ell_{2,p}^3$.

\endproof

\subsection{Full sets in $\ell_{2,p}^3$} 

Let $X$ be a normed linear space and let $S\subseteq X$ be a non-empty set.
We say that $S$ is {\em isometrically  full} in  $X$ if the only isometry in $\operatorname{Isom}_0(X)$ which fixes every point in $S$ is the identity map.

\begin{lem}\label{lem:fullfullaff}
If $S$ has full affine span in $X$ then $S$ is isometrically full in $X$.
\end{lem}

\proof
Suppose that there exists $\phi\in \operatorname{Isom}_0(X)$ such that $\phi(s)=s$ for every $s\in S$. 
Note that $\phi$ is of the form $\phi(x)=Ax+b$, for some linear operator $A$ and $b\in X$. Fix some element $s_0\in S$. Then the operator $A$  also lies in $\operatorname{Isom}_0(X)$ and it is the identity on the linear span of the set $\{s-s_0:s\in S\}$. Since $S$ has full affine span, it follows that $A$ is the identity. Since $b=\phi(s)-s=0$ we see that $\phi$ is the identity map. 
\endproof

Define the restriction map,
\[\rho_S:\mathcal{T}(X)\rightarrow X^S:\eta\mapsto (\eta(s))_{s\in S}.\]
We say that $S$ is {\em full} in $X$ if $\rho_S$ is injective (see \cite{kitlev}).

\begin{prop}
\label{prop:fullsetsl2p3}
Let $p\in (1,\infty)$, $p\not=2$, and let $S$ be a non-empty subset of $\ell_{2,p}^3$.
The following statements are equivalent.
\begin{enumerate}[(i)]
\item $S$ is full in $\ell_{2,p}^3$.
\item $S$ is isometrically full in $\ell_{2,p}^3$.
\item The orthogonal projection of $S$ onto the $xy$-plane contains at least two points. 
\end{enumerate}
\end{prop}

\proof
Let $P_{xy}$ denote the projection of $\ell_{2,p}^3$ onto the $xy$-plane along the $z$-axis. 

$(i)\Leftrightarrow(iii)$
 Suppose that $P_{xy}(S)=\{s\}$. Say $s=(s_1,s_2,0)$ and define 
\[\eta: \ell_{2,p}^3\rightarrow \ell_{2,p}^3,\quad x\mapsto \begin{bmatrix} -x_2\\ x_1\\0 \end{bmatrix}+
\begin{bmatrix} s_2\\ -s_1\\0 \end{bmatrix}.\]
By Theorem \ref{th:infrmdim}, $\eta$ is an infinitesimal rigid motion of $\ell_{2,p}^3$.
Note that $\rho_S(\eta)=0$ and so $S$ is not full in $\ell_{2,p}^3$.

Let us now assume that there exist  $s, r\in S$ such that $P_{xy}(s)\neq P_{xy}(r)$. If $S$ is not full, then there exists a non-zero $\eta\in  \mathcal{T}(\ell_{2,p}^3)$ that satisfies $\eta(s)=\eta(r)=0$. Write $s=(s_1,s_2,s_3)$ and $r=(r_1,r_2,r_3)$. Then, by Theorem \ref{th:infrmdim}, it follows that $(-s_2,s_1,0)=(-r_2,r_1,0)$. Hence 
$P_{xy}(s)=P_{xy}(r)$, a contradiction.

$(ii)\Leftrightarrow(iii)$
Suppose first that $S$ is isometrically full and that the set $P_{xy}(S)=\{P_{xy}(s):s\in S\}$ is a singleton $\{a\}$. Then for every rotation $R$ about the $z$-axis we have
\[T_{a}RT_{-a}s=T_aR(s-a)=T_a(s-a)=s, \quad \forall\,s\in S,\]
a contradiction.

For the converse, suppose there exists $s_1,s_2\in S$ such that $P_{xy}(s_1)\neq P_{xy}(s_2)$.
Note that, by Proposition \ref{prop:orprisom}, it follows that every isometry in $\operatorname{Isom}_0(\ell_{2,p}^3)$  can be written in the form $T_tR$ for some rotation $R$ about the $z$-axis and some translation $T_t$.
Suppose $T_tR(s)=s$ for each $s\in S$ and let $s_1,s_2\in S$ be such that $P_{xy}(s_1)\neq P_{xy}(s_2)$. 
Then $s_1-s_2 = T_tR(s_1)-T_tR(s_2) = R(s_1-s_2)$ and so,
\[P_{xy}(s_1-s_2) = P_{xy}R(s_1-s_2) = RP_{xy}(s_1-s_2).\]
If $R$ is not the identity map then $P_{xy}(s_1-s_2)=0$, a contradiction.
It follows that $T_tR$ is the identity map and so $S$ is full.
\endproof

\begin{rem}
We expect that in any finite dimensional real normed linear space a subset $S$ is full if and only if it is isometrically full, but we are currently unaware of such a proof.
\end{rem}

\subsection{Frameworks in $\ell_{2,p}^3$}
Let $G=(V,E)$ be a finite simple graph. A {\em (bar-joint) framework} in $X$ is a pair $(G,q)$ where $q=(q_v)_{v\in V}\in X^V$ and $q_v\not=q_w$ whenever $vw\in E$.
A {\em subframework} of $(G,q)$ is a framework $(H, q_H)$ where $H=(V(H), E(H))$ is a subgraph of $G$ and $q_H(v) =q(v)$ for all $v\in V(H)$.

A framework $(G,q)$ is said to be {\em full} in $X$ if the set $S=\{q_v:v\in V\}$ is full in $X$.
A framework $(G,q)$ is {\em completely full} in $X$ if it is full in $X$ and every subframework of $(G,q)$ containing at least $2\dim(X)$ vertices is also full in $X$. 

The {\em rigidity map} for $G$ and $X$ is defined by $f_G:X^V\to\mathbb{R}^E$, $x\mapsto (\|x_v-x_w\|)_{vw\in E}$. An {\em infinitesimal flex} of a bar-joint framework $(G,q)$ in $X$ is a vector $m\in X^V$ such that, for each edge $vw\in E$, the directional derivative of the rigidity map $f_G$ in the direction of $m$ vanishes,
\[\lim_{t\to 0} \frac{1}{t}(f_G(q+tm)-f_G(q))=0.\] 

An infinitesimal flex $m\in X^V$ is said to be {\em trivial} if there exists an infinitesimal rigid motion $\eta\in \mathcal{T}(X)$ such that $m_v=\eta(q_v)$ for all $v\in V$. A bar-joint framework $(G,q)$ in $X$ is {\em infinitesimally rigid} if and only if every infinitesimal flex of $(G,q)$ is trivial.

\begin{lem}
\label{lem:mnflex}
Let $(G,q)$ be a bar-joint framework   in $\ell_{2,p}^3$, where $p\in(1,\infty)$. Then a vector $m\in X^V$ is an infinitesimal flex of $(G,q)$ if and only if for each edge $vw\in E$ we have,
 \begin{eqnarray*}
\begin{cases}
(x,y,\frac{\sgn(z)|z|^{p-1}}{d^{p-2}})\cdot (a,b,c)=0, &\,\,\,\, \mbox{if }d\neq 0, \\ 
c=0, &\,\,\,\,  \textrm{otherwise}.
\end{cases}
\end{eqnarray*}
where $q_v-q_w=(x,y,z)$, $m_v-m_w=(a,b,c)$ and $d=(x^2+y^2)^{\frac{1}{2}}$. 
\end{lem}

\proof
For each edge $vw\in E$, consider the function
\[\zeta_{vw}:\mathbb{R}\to \mathbb{R},\quad t\mapsto \|(q_v+tm_v)-(q_w+tm_w)\|_{2,p}.\]
Note that $m$ is an infinitesimal flex for $(G,q)$ if and only if $\zeta_{vw}'(0)=0$ for each edge $vw\in E$. 
As in the proof of Lemma \ref{lem:smooth}, by expressing $\zeta_{vw}$ in the form $\zeta_{vw}=h\circ\left( f+g\right)$ we can show that $\zeta_{vw}$ is differentiable at $0$. 
If $d\neq 0$ then using the chain rule we compute,  
\[\zeta_{vw}'(0)=
(d^p+|z|^p)^{\frac{1}{p}-1}(d^{p-2}(xa+yb)+\sgn(z)|z|^{p-1}c),\]
Rearranging the above we obtain the desired equation. 
If $d=0$, then $z\neq 0$, so we have,
\[\zeta_{vw}'(0)=h'(f(0)+g(0))(f'(0)+g'(0))=\left(|z|^p\right)^{\frac{1-p}{p}}\sgn(z)|z|^{p-1}c=\sgn(z)c.\]
The result follows.
\endproof

\subsection{The rigidity matrix}
We define the {\em rigidity matrix} $R(G,q)$ for a graph $G$ and a vector $q\in\left(\ell_{2,p}^3\right)^V$ to be the $|E|\times 3|V|$ matrix with rows indexed by $E$, columns indexed by $V\times \{1,2,3\}$ and entries defined as follows: Let $vw\in E$. Write $q_v-q_w=(x,y,z)$ and $d=(x^2+y^2)^{\frac{1}{2}}$. When $d\not=0$ then the entries of the row indexed by $vw$ are given by,
\[ \kbordermatrix{
& & & & (v,1) & (v,2) & (v,3) & & & & (w,1) & (w,2) & (w,3) & & & \\
vw 
& 0 & \cdots & 0 
& x 
& y
& \frac{\sgn(z)|z|^{p-1}}{d^{p-2}} 
& 0 & \cdots & 0
& -x 
& -y 
& -\frac{\sgn(z)|z|^{p-1}}{d^{p-2}} 
& 0 & \cdots & 0
}.\]
When $d=0$, then the entries of the row indexed by $vw$ are given by, 
\[ \kbordermatrix{
& & & & (v,1) & (v,2) & (v,3) & & & & (w,1) & (w,2) & (w,3) & & & \\
vw 
& 0 & \cdots & 0 
& 0 
& 0
& z 
& 0 & \cdots & 0
& 0 
& 0 
& -z 
& 0 & \cdots & 0
}.\]

\begin{lem}
\label{l:mixednormrank}
Let $p\in(1,\infty)$, $p\not=2$. A full bar-joint framework $(G,q)$ in $\ell_{2,p}^3$ is infinitesimally rigid if and only if $\rank R(G,q)=3|V|-4$.
\end{lem}

\proof
By Lemma \ref{lem:mnflex}, the kernel of $R(G,q)$ is the linear space of infinitesimal flexes of $(G,q)$. Also, $\rank R(G,q)=3|V|-\dim \ker R(G,q)$.
Since $(G,q)$ is full, by Theorem \ref{th:infrmdim} the infinitesimal rigid motions of  $\ell_{2,p}^3$ induce a $4$-dimensional space of trivial infinitesimal flexes on $(G,q)$.
The result now follows.
\endproof

This gives the following analogue of Theorem~\ref{thm_hypercyclinder_necessary}. 

\begin{thm}
    \label{thm_mixnorms_necessary}
    Let $p\in(1,\infty)$, $p\not=2$. Suppose that $G=(V,E)$ has at least six vertices and that $(G,q)$  is an infinitesimally rigid and completely full framework in $\ell_{2,p}^3$. Furthermore, suppose that for any $e\in E$, $(G-e,q)$ is not infinitesimally rigid. Then $G$ is $(3,4)$-tight. 
\end{thm}

\begin{proof} If $|E|<3|V|-4$, then the rigidity matrix
$R(G,q)$ has rank less than $3|V|-4$. Hence, by Lemma \ref{l:mixednormrank}, $(G,q)$ is
infinitesimally flexible, a contradiction. Now suppose $|E|>3|V|-4$. By Lemma \ref{l:mixednormrank},  
$R(G,q)$ has a non-trivial row dependence and hence there is an edge
whose removal does not decrease the rank of the
rigidity matrix. Thus $G-e$ is still infinitesimally rigid, a contradiction.

Similarly, if there is a non-trivial subgraph $G'=(V',E')$ with $|E'|> 3|V'|-4$, then,
by the simplicity of $G$, $|V'|\geq 6$.
Since $(G,q)$  is completely full, the subframework $(G',q_{G'})$ is full  in $\ell_{2,p}^3$. 
Thus, by Lemma \ref{l:mixednormrank},  the $|E'|\times 3|V'|$ submatrix of $R(G,q)$
corresponding to $G'$ has a non-trivial row dependence. Thus, there is an edge of $G'$ whose
 removal from $G$ leaves the framework infinitesimally rigid, a contradiction.
\end{proof}

It is open as to whether every $(3,4)$-tight graph can be realised as a minimally infinitesimally rigid framework in $\ell_{2,p}^3$, when $p\in (1,\infty)$ and $p\not=2$. However, we will now show that if $G$ is the graph of a doubly braced triangulation, then such a realisation of $G$ always exists.

\subsection{Minimal rigidity of doubly braced triangulations in $\ell_{2,p}^3$}

We first show that the irreducible base graphs given in Theorem~\ref{thm_irred_seven} can be realised as minimally infinitesimally rigid bar-joint frameworks in $\ell_{2,p}^3$ whenever $p\in (1,\infty)$ and $p\not=2$.

\begin{ex}
\label{e:baserigid1}
Consider the base graph $K_6-e$. 
Let $V(K_4-e)=\{s_1,s_2,s_3,s_4\}$ where $e=s_2s_4$ is the deleted edge. To obtain $K_6-e$ we cone $K_4-e$ with a vertex $v_0$ and the resulting graph with another vertex $v_1$. Note that $K_6-e$ is the underlying graph of the irreducible doubly braced triangulations given in Figures \ref{fig_doubly_braced_oct}, \ref{fig_hex_disjoint} and \ref{fig_hex_adjacent} (see also Figure \ref{K_6-e}).

\begin{figure}[ht]
        \centering

        \definecolor{ududff}{rgb}{0,0,0}
    \begin{tikzpicture}[line cap=round,line join=round,>=triangle 45,x=1cm,y=1cm,scale=0.9]
    
    \draw [line width=1.5pt] (-10.81,3.44)-- (-9.77,3.48);
    \draw [line width=1.5pt] (-9.77,3.48)-- (-10.25,2.44);
    \draw [line width=1.5pt] (-10.25,2.44)-- (-10.81,3.44);
		\draw [line width=1.5pt] (-10.25,2.44)-- (-13.39,1.4);
    \draw [line width=1.5pt] (-10.81,3.44)-- (-13.39,1.4);
    \draw [line width=1pt,blue] (-9.77,3.48)-- (-10.31,5.98);
    \draw [line width=1pt, red] (-9.77,3.48)-- (-7.29,1.5);
    \draw [line width=1pt, red] (-10.25,2.44)-- (-7.29,1.5);
    \draw [line width=1pt, red] (-10.81,3.44)-- (-7.29,1.5);
    \draw [line width=1pt,blue] (-10.25,2.44)-- (-10.31,5.98);
		\draw [line width=1pt,blue] (-10.31,5.98)-- (-13.39,1.4);
    \draw [line width=1pt, red] (-13.39,1.4)-- (-7.29,1.5);
    \draw [line width=1pt, red] (-7.29,1.5)-- (-10.31,5.98);
    \draw [line width=1pt,blue] (-10.81,3.44)-- (-10.31,5.98);
    \begin{scriptsize}
    \draw [fill=blue] (-10.31,5.98) circle (1.5pt);
		\draw [blue] (-10.31,6.19) node {$v_0$};
    \draw [fill=red] (-7.29,1.5) circle (1.5pt);
		\draw[red] (-7.05,1.5) node {$v_1$};
		\draw [fill=ududff] (-13.39,1.4) circle (1.5pt);
		\draw[color=ududff] (-13.6,1.4) node {$s_2$};
    \draw [fill=ududff] (-10.81,3.44) circle (1.5pt);
		\draw[color=ududff] (-11,3.54) node {$s_3$};
    \draw [fill=ududff] (-9.77,3.48) circle (1.5pt);
		\draw[color=ududff] (-9.55,3.54) node {$s_4$};
    \draw [fill=ududff] (-10.25,2.44) circle (1.5pt);
		\draw[color=ududff] (-10.25,2.25) node {$s_1$};
    \end{scriptsize}
    \end{tikzpicture}

    \caption{The base graph $K_6-e$.}
        \label{K_6-e}
\end{figure}
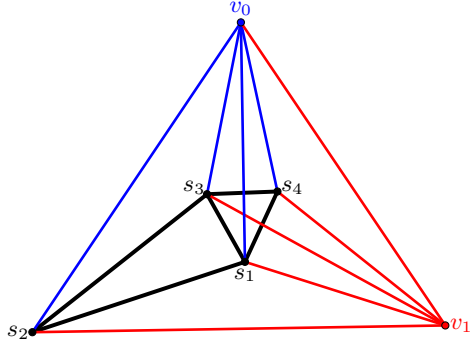

Let $q$ be the following placement: 
\[
s_1=(1,0,0),\quad s_2=(0,1,0),\quad s_3=(-1,0,0), \quad s_4=(0,-1,0),\]
\[v_0=(1,1,1), \quad v_1=(0,0,-1).\]
Then the rigidity matrix is of the form 
\[R(K_6-e,q)=\begin{bmatrix} A & 0\\ \ast & D \end{bmatrix},\]
where the submatrix $A$ contains the entries arising from the $x$ and $y$ coordinates of the edges of $K_4-e$ and is of the form
\[A =\kbordermatrix{ 
& (s_1,1) & (s_1,2)  & (s_2,1) & (s_2,2) & (s_3,1) & (s_3,2) & (s_4,1) & (s_4,2) \\
s_1s_2 & 1 & -1  & -1 & 1 &  0 & 0 &  0 & 0  \\
s_1s_3 & 2 & 0  & 0 & 0 & -2 & 0 & 0 & 0  \\
s_1s_4 & 1 & 1  & 0 & 0 & 0 & 0 & 1 & 1   \\
s_2s_3 & 0 & 0  & 1 & 1 & -1 & -1 &  0 & 0   \\
s_3s_4 & 0 & 0  & 0 & 0 & -1 & 1 & 1 & -1   \\
}\]
and $D$, which lies in $M_{9\times 10}(\mathbb{R})$, is given below;
\[D=\kbordermatrix{ 
& (s_1,3) & (s_2,3) & (s_3,3) & (s_4,3) & (v_0,1) & (v_0,2) & (v_0,3) & (v_1,1) & (v_1,2) & (v_1,3)\\
s_1v_0 & -1 & 0 & 0 & 0 & 0 & 1 & 1 & 0 & 0 & 0 \\
s_2v_0 & 0 & -1 & 0 & 0 & 1 & 0 & 1 & 0 & 0 & 0  \\
s_3v_0 & 0 & 0 & -\sqrt{5}^{2-p} & 0 & 2 & 1 & \sqrt{5}^{2-p} & 0 & 0 & 0  \\
s_4v_0 & 0 & 0 & 0 & -\sqrt{5}^{2-p} & 1 & 2 & \sqrt{5}^{2-p} & 0 & 0 & 0 \\
s_1v_1 & 1 & 0 & 0 & 0 & 0 & 0 & 0 & -1 & 0 & -1  \\
s_2v_1 & 0 & 1 & 0 & 0 & 0 & 0 & 0 & 0 & -1 & -1  \\
s_3v_1 & 0 & 0 & 1 & 0 & 0 & 0 & 0 & 1 & 0 & -1  \\
s_4v_1 & 0 & 0 & 0 & 1 & 0 & 0 & 0 & 0 & 1 & -1  \\
v_0v_1 & 0 & 0 & 0 & 0 & 1 & 1 & \sqrt{2}^p & -1 & -1 & -\sqrt{2}^p  
}\]

Since the rows of the matrix $A$ are evidently linearly independent, it suffices to show that the rows of the matrix $D$ are also linearly independent. In the remaining argument, the row operations will be indicated with the standard notation. For example, the fifth row of the matrix $D_1$ below is the sum of the first and the fifth row of the matrix $D$, so we write $R_5=r_5+r_1$.  
\[D_1=\resizebox{0.95\hsize}{!}{$\kbordermatrix{ 
& (s_1,3) & (s_2,3) & (s_3,3) & (s_4,3) & (v_0,1) & (v_0,2) & (v_0,3) & (v_1,1) & (v_1,2) & (v_1,3)\\
r_1 & -1 & 0 & 0 & 0 & 0 & 1 & 1 & 0 & 0 & 0 \\
r_2 & 0 & -1 & 0 & 0 & 1 & 0 & 1 & 0 & 0 & 0  \\
r_3 & 0 & 0 & -\sqrt{5}^{2-p} & 0 & 2 & 1 & \sqrt{5}^{2-p} & 0 & 0 & 0  \\
r_4 & 0 & 0 & 0 & -\sqrt{5}^{2-p} & 1 & 2 & \sqrt{5}^{2-p} & 0 & 0 & 0 \\
R_5= r_5+r_1 & 0 & 0 & 0 & 0 & 0 & 1 & 1 & -1 & 0 & -1  \\
R_6= r_6+r_2 & 0 & 0 & 0 & 0 & 1 & 0 & 1 & 0 & -1 & -1  \\
R_7= \sqrt{5}^{2-p}r_7+r_3 & 0 & 0 & 0 & 0 & 2 & 1 & \sqrt{5}^{2-p} & \sqrt{5}^{2-p} & 0 & -\sqrt{5}^{2-p}  \\
R_8= \sqrt{5}^{2-p}r_8+r_4 & 0 & 0 & 0 & 0 & 1 & 2 & \sqrt{5}^{2-p} & 0 & \sqrt{5}^{2-p} & -\sqrt{5}^{2-p}  \\
r_9 & 0 & 0 & 0 & 0 & 1 & 1 & \sqrt{2}^p & -1 & -1 & -\sqrt{2}^p  
}$}\]
 It is evident now that the first four rows of this matrix are linearly independent, so we may focus on the submatrix:
\[D_2=\kbordermatrix{ 
&  (v_0,1) & (v_0,2) & (v_0,3) & (v_1,1) & (v_1,2) & (v_1,3)\\
r_1  & 0 & 1 & 1 & -1 & 0 & -1  \\
r_2  & 1 & 0 & 1 & 0 & -1 & -1  \\
r_3 &2 & 1 & \sqrt{5}^{2-p} & \sqrt{5}^{2-p} & 0 & -\sqrt{5}^{2-p}  \\
r_4  & 1 & 2 & \sqrt{5}^{2-p} & 0 & \sqrt{5}^{2-p} & -\sqrt{5}^{2-p}  \\
r_5  & 1 & 1 & \sqrt{2}^p & -1 & -1 & -\sqrt{2}^p  
}\]
Next, we eliminate the matrix elements below the first entry in the main diagonal of the above matrix, and get the equivalent matrix 
\[D_3=\kbordermatrix{ 
&  (v_0,1) & (v_0,2) & (v_0,3) & (v_1,1) & (v_1,2) & (v_1,3)\\
R_1=r_2  & 1 & 0 & 1 & 0 & -1 & -1   \\
R_2=r_1  & 0 & 1 & 1 & -1 & 0 & -1  \\
R_3= r_3-2r_2 &0 & 1 & -2+\sqrt{5}^{2-p} & \sqrt{5}^{2-p} & 2 &2-\sqrt{5}^{2-p}  \\
R_4=r_4-r_2  & 0 & 2 & -1+\sqrt{5}^{2-p} & 0 & 1+\sqrt{5}^{2-p} & 1-\sqrt{5}^{2-p}  \\
R_5=r_5-r_2  & 0 & 1 & -1+\sqrt{2}^p & -1 & 0 & 1-\sqrt{2}^p  
}\]
 Thus, we can remove the first row and the first column. Working in a similar manner we obtain
\[D_4=\kbordermatrix{ 
&   (v_0,2) & (v_0,3) & (v_1,1) & (v_1,2) & (v_1,3)\\
r_1  & 1 & 1 & -1 & 0 & -1  \\
R_2=r_2+r_1  & 0 & 3-\sqrt{5}^{2-p} & -1-\sqrt{5}^{2-p} & -2 &-3+\sqrt{5}^{2-p}  \\
R_3=r_3+2r_1  & 0 & 3-\sqrt{5}^{2-p} & -2 & -1-\sqrt{5}^{2-p} &-3+\sqrt{5}^{2-p}  \\
R_4=r_4-r_1  & 0 & -2+\sqrt{2}^p & 0 & 0 & 2-\sqrt{2}^p  
}.\]
Note that the last row of the matrix $D_4$ becomes zero for $p=2$.
For $p\neq 2$, we remove again the first row and the first column and rearrange 
\[D_5=\kbordermatrix{ 
&  (v_0,3) & (v_1,1) & (v_1,2) & (v_1,3)\\
R_1=r_3 & -2+\sqrt{2}^p & 0 & 0 & 2-\sqrt{2}^p   \\
R_2=r_1-x r_3  & 0 & -1-\sqrt{5}^{2-p} & -2 &0  \\
R_3=r_2-xr_3 & 0 & -2 & -1-\sqrt{5}^{2-p} &0  
}\]
where $x=\frac{3-\sqrt{5}^{2-p}}{-2+\sqrt{2}^p}$. Thus, it suffices to show that the matrix
\[D_6=\kbordermatrix{ 
& (v_1,1) & (v_1,2) \\
R_1=r_2 & -1-\sqrt{5}^{2-p} & -2 \\
R_2=r_3 & -2 & -1-\sqrt{5}^{2-p}   
}\] 
has linearly independent rows, which is true for every $p\neq 2$.
\end{ex}

\begin{ex}
\label{e:baserigid2}
Consider the base graph $K_5 \cup_{K_3} K_5$. This graph can be obtained from $K_4-e$ by  repeatedly adding three degree 4 vertices (see Figure~\ref{fig_ireed_seven K53K5}). We denote  $V(K_4-e)=\{s_1,s_2,s_3,s_4\}$  and the extra vertices by $v_1,v_2,v_3$. Note that $K_5 \cup_{K_3} K_5$ is the underlying graph of the irreducible doubly braced triangulations  given in Figures \ref{fig_ireed_seven} and \ref{fig_ireed_seven_non}. The two $K_5$ subgraphs will be described by the respective vertex sets $\{s_1,s_2,s_3,v_1,v_2\}$ and 
$\{s_1,s_3,s_4,v_1,v_3\}$. The intersection of those subgraphs is the graph $K_3$ indicated by the dashed edges in Figure~\ref{fig_ireed_seven K53K5}.

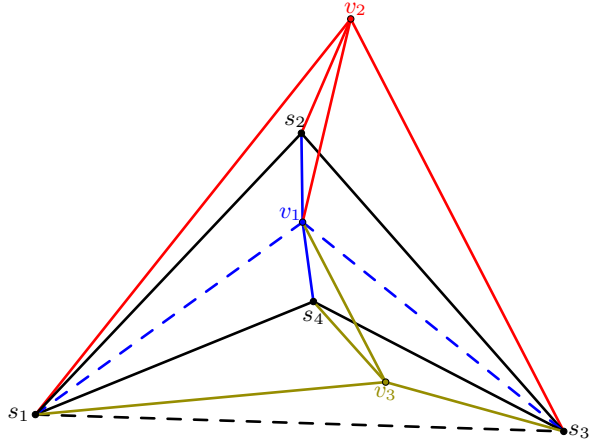
\begin{figure}[ht]
    \definecolor{ududff}{rgb}{0,0,0}
    \begin{tikzpicture}[line cap=round,line join=round,>=triangle 45,x=1cm,y=1cm,scale=0.8]
    \draw [line width=1pt,dash pattern=on 5pt off 5pt] (-3.32,-0.6)-- (5.46,-0.88);
    \draw [line width=1pt,red] (5.46,-0.88)-- (1.92,5.98);
    \draw [line width=1pt, red] (1.92,5.98)-- (-3.32,-0.6);
    \draw [line width=1pt,blue] (1.3,1.28)-- (1.12,2.6);
    \draw [line width=1pt,blue] (1.12,2.6)-- (1.1,4.08);
    \draw [line width=1pt,red] (1.1,4.08)-- (1.92,5.98);
    \draw [line width=1pt] (1.3,1.28)-- (-3.32,-0.6);
    \draw [line width=1pt,blue,dash pattern=on 5pt off 5pt] (1.12,2.6)-- (-3.32,-0.6);
    \draw [line width=1pt] (1.1,4.08)-- (-3.32,-0.6);
    \draw [line width=1pt] (1.1,4.08)-- (5.46,-0.88);
    \draw [line width=1pt,blue,dash pattern=on 5pt off 5pt] (1.12,2.6)-- (5.46,-0.88);
    \draw [line width=1pt] (1.3,1.28)-- (5.46,-0.88);
    \draw [line width=1pt, olive] (1.3,1.28)-- (2.5,-0.06);
    \draw [line width=1pt, olive] (2.5,-0.06)-- (-3.32,-0.6);
    \draw [line width=1pt, olive] (2.5,-0.06)-- (5.46,-0.88);
    \draw [line width=1pt,red] (1.12,2.6)-- (1.92,5.98);
    \draw [line width=1pt, olive] (1.12,2.6)-- (2.5,-0.06);
		    \begin{scriptsize}
    \draw [fill=ududff] (-3.32,-0.6) circle (1.5pt);
    \draw [fill=ududff] (5.46,-0.88) circle (1.5pt);
    \draw [fill=ududff] (1.3,1.28) circle (1.5pt);
    \draw [fill=ududff] (1.1,4.08) circle (1.5pt);
		\draw [fill=blue] (1.12,2.6) circle (1.5pt);
    \draw [fill=red] (1.92,5.98) circle (1.5pt);
		\draw [fill=olive] (2.5,-0.06) circle (1.5pt);
		\draw[color=ududff] (-3.6,-0.6) node {$s_1$};
		\draw[color=ududff] (5.72,-0.88) node {$s_3$};
		\draw[color=ududff] (1.3,1.02) node {$s_4$};
		\draw[color=ududff] (1,4.28) node {$s_2$};
		\draw[color=blue] (0.92,2.74) node {$v_1$};
		\draw[color=red] (2,6.15) node {$v_2$};
		\draw[color=olive] (2.5,-0.25) node {$v_3$};
    \end{scriptsize}
    \end{tikzpicture}
    \caption{The base graph $K_5 \cup_{K_3} K_5$.}
    \label{fig_ireed_seven K53K5}
\end{figure}

The placement $q$ of $V(K_4-e)$ is the same as in Example \ref{e:baserigid1}, while the vertices $v_1$, $v_2$, $v_3$  are placed at the respective points $(0,0,1)$, $(-1,1,-1)$, $(-1,-1,-1)$. Following the same procedure as in the previous example, the rigidity matrix of the framework $(K_5 \cup_{K_3} K_5,q)$ is a lower triangular block matrix
\[R(K_5 \cup_{K_3} K_5,q)=\begin{bmatrix} A & 0\\ \ast & D \end{bmatrix},\]
where the submatrix $A$ contains the entries arising from the $x$ and $y$ coordinates of the edges of $K_4-e$, and $D$ is the following matrix

\[D=\resizebox{0.95\hsize}{!}{$\kbordermatrix{ 
& (s_1,3) & (s_2,3) & (s_3,3) & (s_4,3) & (v_1,1) & (v_1,2) & (v_1,3) & (v_2,1) & (v_2,2) & (v_2,3)& (v_3,1) & (v_3,2) & (v_3,3)\\
s_1v_1 & -1 & 0 & 0 & 0 & -1 & 0 & 1 & 0 & 0 & 0 & 0 & 0 & 0 \\
s_2v_1 & 0 & -1 & 0 & 0 & 0 & -1 & 1 & 0 & 0 & 0 & 0 & 0 & 0  \\
s_3v_1 & 0 & 0 & -1 & 0 & 1 & 0 & 1 & 0 & 0 & 0 & 0 & 0 & 0  \\
s_4v_1 & 0 & 0 & 0 & -1 & 0 & 1 & 1 & 0 & 0 & 0 & 0 & 0 & 0 \\
s_1v_2 & \sqrt{5}^{2-p} & 0 & 0 & 0 & 0 & 0 & 0 & -2 & 1 & -\sqrt{5}^{2-p} & 0 & 0 & 0  \\
s_2v_2 & 0 & 1 & 0 & 0 & 0 & 0 & 0 & -1 & 0 & -1 & 0 & 0 & 0  \\
s_3v_2 & 0 & 0 & 1 & 0 & 0 & 0 & 0 & 0 & 1 & -1 & 0 & 0 & 0  \\
v_1v_2 & 0 & 0 & 0 & 0 & 1 & -1 & \sqrt{2}^{p} & -1 & 1 & -\sqrt{2}^{p} & 0 & 0 & 0  \\
s_1v_3 & \sqrt{5}^{2-p} & 0 & 0 & 0 & 0 & 0 & 0 & 0 & 0 & 0 & -2 & -1 & -\sqrt{5}^{2-p}  \\
s_3v_3 & 0 & 0 & 1 & 0 & 0 & 0 & 0 & 0 & 0 & 0 & 0 & -1 & -1  \\
s_4v_3 & 0 & 0 & 0 & 1 & 0 & 0 & 0 & 0 & 0 & 0 & -1 & 0 & -1  \\
v_1v_3 & 0 & 0 & 0 & 0 & 1 & 1 & \sqrt{2}^{p} & 0 & 0 & 0 & -1 & -1 & -\sqrt{2}^{p}  \\
}$}\]

It suffices to show again that the rows of the matrix $D$ are linearly independent. Performing row operations  in order to eliminate the subdiagonal elements of the first four columns, we obtain the equivalent matrix 
\[\begin{bmatrix} D_1 & D_2\\  0 & E \end{bmatrix}\]
where the blocks $D_1, D_2$ form the first 4 rows of the matrix $D$ and $E$ is given by

\[E=\kbordermatrix{ 
 & (v_1,1) & (v_1,2) & (v_1,3) & (v_2,1) & (v_2,2) & (v_2,3)& (v_3,1) & (v_3,2) & (v_3,3)\\
r_1  & -\sqrt{5}^{2-p} & 0 & \sqrt{5}^{2-p} & -2 & 1 & -\sqrt{5}^{2-p} & 0 & 0 & 0  \\
r_2  & 0 & -1 & 1 & -1 & 0 & -1 & 0 & 0 & 0  \\
r_3  & 1 & 0 & 1 & 0 & 1 & -1 & 0 & 0 & 0  \\
r_4  & 1 & -1 & \sqrt{2}^{p} & -1 & 1 & -\sqrt{2}^{p} & 0 & 0 & 0  \\
r_5  & -\sqrt{5}^{2-p} & 0 & \sqrt{5}^{2-p} & 0 & 0 & 0 & -2 & -1 & -\sqrt{5}^{2-p}  \\
r_6  & 1 & 0 & 1 & 0 & 0 & 0 & 0 & -1 & -1  \\
r_7  & 0 & 1 & 1  & 0 & 0 & 0 & -1 & 0 & -1  \\
r_8  & 1 & 1 & \sqrt{2}^{p} & 0 & 0 & 0 & -1 & -1 & -\sqrt{2}^{p}  
}\]

We work now simultaneously on 2 different blocks of $E$, the first one is formed by the first 4 rows of $E$ and the second one is given from the remaining rows.

\[E_1=\resizebox{0.95\hsize}{!}{$\kbordermatrix{ 
 & (v_1,1) & (v_1,2) & (v_1,3) & (v_2,1) & (v_2,2) & (v_2,3)& (v_3,1) & (v_3,2) & (v_3,3)\\
R_1=-r_2 & 0 & 1 & -1 & 1 & 0 & 1 & 0 & 0 & 0  \\
R_2=r_1-2r_2  & -\sqrt{5}^{2-p} & 2 & -2+\sqrt{5}^{2-p} & 0 & 1 & 2-\sqrt{5}^{2-p} & 0 & 0 & 0  \\
r_3 & 1 & 0 & 1 & 0 & 1 & -1 & 0 & 0 & 0  \\
R_4=r_4-r_2  & 1 & 0 & -1+\sqrt{2}^{p} & 0 & 1 & 1-\sqrt{2}^{p} & 0 & 0 & 0  \\
R_5=-r_7  & 0 & -1 & -1  & 0 & 0 & 0 & 1 & 0 & 1  \\
R_6=r_6  & 1 & 0 & 1 & 0 & 0 & 0 & 0 & -1 & -1  \\
R_7=r_5-2r_7  &-\sqrt{5}^{2-p} & -2 & -2+\sqrt{5}^{2-p} & 0 & 0 & 0 & 0 & -1 & 2-\sqrt{5}^{2-p}  \\
R_8=r_8-r_7  & 1 & 0 & -1+\sqrt{2}^{p} & 0 & 0 & 0 & 0 & -1 & 1-\sqrt{2}^{p}  \\
}$}\]
Hence we may remove the  first and the fifth row and the columns $(v_2,1)$ and $(v_3,1)$ to obtain the matrix $E_2$
\[E_2=\kbordermatrix{ 
 & (v_1,1) & (v_1,2) & (v_1,3) & (v_2,2) & (v_2,3) & (v_3,2) & (v_3,3)\\
r_1  & -\sqrt{5}^{2-p} & 2 & -2+\sqrt{5}^{2-p} & 1 & 2-\sqrt{5}^{2-p} & 0 & 0  \\
r_2  & 1 & 0 & 1 & 1 & -1 & 0 & 0 \\
r_3   & 1 & 0 & -1+\sqrt{2}^{p} & 1 & 1-\sqrt{2}^{p} & 0 & 0  \\
r_4  & 1 & 0 & 1 & 0 & 0 & -1 & -1  \\
r_5  &-\sqrt{5}^{2-p} & -2 & -2+\sqrt{5}^{2-p} & 0 & 0 & -1 & 2-\sqrt{5}^{2-p}   \\
r_6 & 1 & 0 & -1+\sqrt{2}^{p} & 0 & 0  & -1 & 1-\sqrt{2}^{p}  \\
}\]

We continue with the following row operations:

\[E_3=\resizebox{0.9\hsize}{!}{$\kbordermatrix{ 
& (v_1,1) & (v_1,2) & (v_1,3) & (v_2,2) & (v_2,3) & (v_3,2) & (v_3,3)\\
R_1=r_2  & 1 & 0 & 1 & 1 & -1 & 0 & 0\\
R_2=r_1-r_2  & -1-\sqrt{5}^{2-p} & 2 & -3+\sqrt{5}^{2-p}  & 0 & 3-\sqrt{5}^{2-p} & 0 & 0  \\
R_3=r_3-r_2  & 0 & 0 & -2+\sqrt{2}^{p} & 0 & 2-\sqrt{2}^{p}  & 0 & 0  \\
R_4=r_4  & 1 & 0 & 1 & 0 & 0 & -1 & -1  \\
R_5=r_5-r_4  & -1-\sqrt{5}^{2-p} & -2 & -3+\sqrt{5}^{2-p} & 0 & 0  & 0 & 3-\sqrt{5}^{2-p}  \\
R_6=r_6-r_4  & 0 & 0 & -2+\sqrt{2}^{p} & 0 & 0  & 0 & 2-\sqrt{2}^{p}  \\
}$}\]
Again note that for $p=2$ all the entries of the rows $R_3$ and $R_6$ of $E_3$ are equal to zero, so the matrix fails to have independent rows. For $p\neq 2$, it suffices to show that the rows of the matrix $E_4$, given below, are linearly independent.
\[E_4=\kbordermatrix{ 
& (v_1,1) & (v_1,2) & (v_1,3) & (v_2,3)  & (v_3,3)\\
r_1  & -1-\sqrt{5}^{2-p} & 2 & -3+\sqrt{5}^{2-p}  & 3-\sqrt{5}^{2-p}  & 0  \\
r_2 & 0 & 0 & -2+\sqrt{2}^{p}  & 2-\sqrt{2}^{p}   & 0  \\
r_3 & -1-\sqrt{5}^{2-p} & -2 & -3+\sqrt{5}^{2-p} & 0  & 3-\sqrt{5}^{2-p}  \\
r_4 & 0 & 0 & -2+\sqrt{2}^{p} & 0 & 2-\sqrt{2}^{p}  \\
}.\]
Define again $x=\frac{3-\sqrt{5}^{2-p}}{-2+\sqrt{2}^p}$. Since the equivalent matrix 
\[E_5=\kbordermatrix{ 
& (v_1,1) & (v_1,2) & (v_1,3) & (v_2,3)  & (v_3,3)\\
R_1=r_1+xr_2  &  -1-\sqrt{5}^{2-p} & 2 & 0  & 0  & 0  \\
r_2 & 0 & 0 & -2+\sqrt{2}^{p}  & 2-\sqrt{2}^{p}   & 0  \\
R_3=r_3+xr_4 & -1-\sqrt{5}^{2-p} & -2 & 0 & 0  & 0  \\
r_4 & 0 & 0 & -2+\sqrt{2}^{p} & 0 & 2-\sqrt{2}^{p}  \\
}.\]
has evidently linearly independent rows, it follows that the framework $(K_5 \cup_{K_3} K_5,q)$ is infinitesimally rigid.

\end{ex}

We now recall the following result.

\begin{prop}{\cite[Proposition 4.7]{dkn}}
\label{p:vsrigid}
Let $X$ be a strictly convex and smooth finite dimensional real normed linear space with dimension $d$.
Suppose $G'$ is a graph which is obtained from $G$ by a $d$-dimensional vertex splitting move. 
If there exists $q$ such that $(G,q)$ is (minimally) infinitesimally rigid in $X$ then there exists $q'$ such that $(G',q')$ is (minimally) infinitesimally rigid in $X$.
\end{prop}

A framework $(G,q)$ in $\ell_{2,p}^3$ is said to be {\em regular}  if the function  $$r_G:\left(\ell_{2,p}^3\right)^V\to \mathbb{N},\quad x\mapsto \rank R(G,x),$$ achieves its maximum value at $q$. 
Note that if $(G,q)$ is infinitesimally rigid in $\ell_{2,p}^3$ for some regular placement $q$ then every regular placement of $G$ in $\ell_{2,p}^3$ is infinitesimally rigid.
In this case, we say that the graph $G$ is rigid in $\ell_{2,p}^3$.

\begin{thm}
\label{t:norms}
 Let \( G \) be the graph of a doubly braced triangulation and let $p\in(1,\infty)$, $p\not=2$. 
		Then  \( G  \) is (minimally) rigid in $\ell_{2,p}^3$.
\end{thm}

\proof
Let $G$ be the graph of a doubly braced triangulation. By Theorem~\ref{thm_irred_seven},
  $G$ can be constructed from the graph of one of the irreducible doubly braced triangulations by a sequence of 3-dimensional vertex splitting moves.
Examples \ref{e:baserigid1} and \ref{e:baserigid2} show that the  graphs of these irreducible doubly braced triangulations have an infinitesimally rigid placement in $\ell_{2,p}^3$. (In fact these placements are minimally infinitesimally rigid since they have exactly $3|V|-4$ edges.)  By Lemma \ref{lem:smooth}, $\ell_{2,p}^3$ is strictly convex and smooth  for all $p\in(1,\infty)$. Thus the result follows from Proposition \ref{p:vsrigid}.
\endproof


\printbibliography

\end{document}